\documentclass[12pt]{amsart}
\usepackage{amsmath,amssymb,amsbsy,amsfonts,latexsym,amsopn,amstext,cite,
                                               amsxtra,euscript,amscd,bm}
\usepackage{url}

\usepackage{mathrsfs}

\usepackage{color}
\usepackage[colorlinks,linkcolor=blue,anchorcolor=blue,citecolor=blue,backref=page]{hyperref}
\usepackage{color}
\usepackage{graphics,epsfig}
\usepackage{graphicx}
\usepackage{float}
\usepackage{epstopdf}
\hypersetup{breaklinks=true}

\usepackage[np]{numprint}
\npdecimalsign{\ensuremath{.}}

\usepackage{bibentry}

\usepackage[english]{babel}
\usepackage{mathtools}
\usepackage{todonotes}
\usepackage{url}
\usepackage[colorlinks,linkcolor=blue,anchorcolor=blue,citecolor=blue,backref=page]{hyperref}

\usepackage[norefs,nocites]{refcheck}

\DeclareMathOperator*\uplim{\overline{lim}}

\begin{document}

\newtheorem{theorem}{Theorem}
\newtheorem{lemma}[theorem]{Lemma}
\newtheorem{claim}[theorem]{Claim}
\newtheorem{cor}[theorem]{Corollary}
\newtheorem{conj}[theorem]{Conjecture}
\newtheorem{prop}[theorem]{Proposition}
\newtheorem{definition}[theorem]{Definition}
\newtheorem{question}[theorem]{Question}
\newtheorem{example}[theorem]{Example}
\newcommand{\hh}{{{\mathrm h}}}
\newtheorem{remark}[theorem]{Remark}

\numberwithin{equation}{section}
\numberwithin{theorem}{section}
\numberwithin{table}{section}
\numberwithin{figure}{section}

\def\sssum{\mathop{\sum\!\sum\!\sum}}
\def\ssum{\mathop{\sum\ldots \sum}}
\def\iint{\mathop{\int\ldots \int}}

\newcommand{\diam}{\operatorname{diam}}
\newcommand{\dimloc}{\operatorname{\dim_{loc}}}

\def\squareforqed{\hbox{\rlap{$\sqcap$}$\sqcup$}}
\def\qed{\ifmmode\squareforqed\else{\unskip\nobreak\hfil
\penalty50\hskip1em \nobreak\hfil\squareforqed
\parfillskip=0pt\finalhyphendemerits=0\endgraf}\fi}%%

%  use the AMS-Euler Fraktur fonts
%%%%%%%%%%%%%%%%%%%%%%%%%%%%%%%%%%
\newfont{\teneufm}{eufm10}
\newfont{\seveneufm}{eufm7}
\newfont{\fiveeufm}{eufm5}
%%%%%%%%%%%%%%%%%%%%%%%%%%%%%%%%%
%
%  allow automatic size selection in math mode
%
%%%%%%%%%%%%%%%%%%%%%%%%%%%%%%%%%
\newfam\eufmfam
     \textfont\eufmfam=\teneufm
\scriptfont\eufmfam=\seveneufm
     \scriptscriptfont\eufmfam=\fiveeufm
%%%%%%%%%%%%%%%%%%%%%%%%%%%%%%%%%
%
%  \frak works on a single symbol at a time...
%
\def\frak#1{{\fam\eufmfam\relax#1}}

\newcommand{\bflambda}{{\boldsymbol{\lambda}}}
\newcommand{\bfmu}{{\boldsymbol{\mu}}}
\newcommand{\bfxi}{{\boldsymbol{\eta}}}
\newcommand{\bfrho}{{\boldsymbol{\rho}}}

\def\eps{\varepsilon}

\def\fK{\mathfrak K}
\def\fT{\mathfrak{T}}
\def\fL{\mathfrak L}
\def\fR{\mathfrak R}
\def\fI{\mathfrak I}

\def\sI{\mathsf I} 

\def\fA{{\mathfrak A}}
\def\fB{{\mathfrak B}}
\def\fC{{\mathfrak C}}
\def\fM{{\mathfrak M}}
\def\fQ{{\mathfrak Q}}
\def\fS{{\mathfrak  S}}
\def\fU{{\mathfrak U}}

 \def\sfE {\mathsf {E}}
 \def\sfM {\mathsf {M}}
\def\T {\mathsf {T}}
\def\TTT {\mathsf {T}}
\def\Tor{\mathsf{T}_d}
\def\Tore{\widetilde{\mathrm{T}}_{d} }

\def\sM {\mathsf {M}}
\def\sS{\mathsf {S}}

\def\ss{\mathsf {s}}

\def\Kmnd{\cK_d(m,n)}
\def\Kmnp{\cK_p(m,n)}
\def\Kmnq{\cK_q(m,n)}

\def \balpha{\bm{\alpha}}
\def \bbeta{\bm{\beta}}
\def \bgamma{\bm{\gamma}}
\def \bdelta{\bm{\delta}}
\def \bzeta{\bm{\zeta}}
\def \blambda{\bm{\lambda}}
\def \bchi{\bm{\chi}}
\def \bphi{\bm{\varphi}}
\def \bpsi{\bm{\psi}}
\def \bnu{\bm{\nu}}
\def \bomega{\bm{\omega}}

\def \bell{\bm{\ell}}

\def\eqref#1{(\ref{#1})}

\def\vec#1{\mathbf{#1}}

\newcommand{\abs}[1]{\left| #1 \right|}

\def\Zq{\mathbb{Z}_q}
\def\Zqx{\mathbb{Z}_q^*}
\def\Zd{\mathbb{Z}_d}
\def\Zdx{\mathbb{Z}_d^*}
\def\Zf{\mathbb{Z}_f}
\def\Zfx{\mathbb{Z}_f^*}
\def\Zp{\mathbb{Z}_p}
\def\Zpx{\mathbb{Z}_p^*}
\def\cM{\mathcal M}
\def\cE{\mathcal E}
\def\cH{\mathcal H}

\def\le{\leqslant}

\def\ge{\geqslant}

\def\sfB{\mathsf {B}}
\def\sfC{\mathsf {C}}
\def\sfE{\mathsf {E}}
\def\L{\mathsf {L}}
\def\FF{\mathsf {F}}

\def\sE {\mathscr{E}}
\def\sS {\mathscr{S}}
\def\sF {\mathscr{F}}
\def\sB {\mathscr{B}}
%%%%%%%%%%%%%%%%%%%%%%%%%
% Alphabet calligraphie %
%%%%%%%%%%%%%%%%%%%%%%%%%
\def\cA{{\mathcal A}}
\def\cB{{\mathcal B}}
\def\cC{{\mathcal C}}
\def\cD{{\mathcal D}}
\def\cE{{\mathcal E}}
\def\cF{{\mathcal F}}
\def\cG{{\mathcal G}}
\def\cH{{\mathcal H}}
\def\cI{{\mathcal I}}
\def\cJ{{\mathcal J}}
\def\cK{{\mathcal K}}
\def\cL{{\mathcal L}}
\def\cM{{\mathcal M}}
\def\cN{{\mathcal N}}
\def\cO{{\mathcal O}}
\def\cP{{\mathcal P}}
\def\cQ{{\mathcal Q}}
\def\cR{{\mathcal R}}
\def\cS{{\mathcal S}}
\def\cT{{\mathcal T}}
\def\cU{{\mathcal U}}
\def\cV{{\mathcal V}}
\def\cW{{\mathcal W}}
\def\cX{{\mathcal X}}
\def\cY{{\mathcal Y}}
\def\cZ{{\mathcal Z}}
\newcommand{\rmod}[1]{\: \mbox{mod} \: #1}

\def\cg{{\mathcal g}}

\def\vy{\mathbf y}
\def\vr{\mathbf r}
\def\vx{\mathbf x}
\def\va{\mathbf a}
\def\vb{\mathbf b}
\def\vc{\mathbf c}
\def\ve{\mathbf e}
\def\vh{\mathbf h}
\def\vk{\mathbf k}
\def\vm{\mathbf m}
\def\vz{\mathbf z}
\def\vu{\mathbf u}
\def\vv{\mathbf v}

\def \btau{\bm{\tau}}

\def\e{{\mathbf{\,e}}}
\def\ep{{\mathbf{\,e}}_p}
\def\eq{{\mathbf{\,e}}_q}

\def\Tr{{\mathrm{Tr}}}
\def\Nm{{\mathrm{Nm}}}

 \def\SS{{\mathbf{S}}}

\def\lcm{{\mathrm{lcm}}}

 \def\0{{\mathbf{0}}}

\def\({\left(}
\def\){\right)}
\def\l|{\left|}
\def\r|{\right|}
\def\fl#1{\left\lfloor#1\right\rfloor}
\def\rf#1{\left\lceil#1\right\rceil}
\def\sumstar#1{\mathop{\sum\vphantom|^{\!\!*}\,}_{#1}}

\def\mand{\qquad \mbox{and} \qquad}

\def\tred#1{\begin{color}{red}{{#1}}\end{color}}

%%%%%%%%%%%%%%%%%%%%%%%%%%%%%%%%%%%%%%%%%%%%%%%%%%%%%%%%
%%%%%%%%%%%%%%%%%%%%%%%%%%%%%%%%%%%%%%%%%%%%%%%%%%%%%%%%
%%%%%%%%%%%%%%%%%%%%%%%%%%%%%%%%%%%%%%%%%%%%%%%%%%%%%%%%
%%%%%%%%%%%%%%%%%%%%%%%%%%%%%%%%%%%%%%%%%%%%%%%%%%%%%%%%

%%%%%%%  END OF STANDARD STUFF %%%%%%%%%

%%%%%%%%%%%%%%%%%%%%%%%%%%%%%%%%%%%%%%%%%%%%%%%%%%%%%%%%
%%%%%%%%%%%%%%%%%%%%%%%%%%%%%%%%%%%%%%%%%%%%%%%%%%%%%%%%
%%%%%%%%%%%%%%%%%%%%%%%%%%%%%%%%%%%%%%%%%%%%%%%%%%%%%%%%
%%%%%%%%%%%%%%%%%%%%%%%%%%%%%%%%%%%%%%%%%%%%%%%%%%%%%%%
%%%%%%%%%%%
%%% Spell

\hyphenation{re-pub-lished}

\mathsurround=1pt

\def\bfdefault{b}

\def \F{{\mathbb F}}
\def \K{{\mathbb K}}
\def \N{{\mathbb N}}
\def \Z{{\mathbb Z}}
\def \P{{\mathbb P}}
\def \Q{{\mathbb Q}}
\def \R{{\mathbb R}}
\def \C{{\mathbb C}}
\def\Fp{\F_p}
\def \fp{\Fp^*}

 \DeclarePairedDelimiter{\ceil}{\lceil}{\rceil}

 \def \xbar{\overline x}

\title[Metric theory of Weyl sums]{Metric theory of Weyl sums}

 \author[C. Chen] {Changhao Chen}

\address{Department of Mathematics, The Chinese University of Hong Kong, Shatin, Hong Kong}
\email{changhao.chenm@gmail.com}

 \author[B. Kerr] {Bryce Kerr}
\address{Department of Mathematics and Statistics,
University of Turku,  FI-20014, Finland}
\email{bryce.kerr@utu.fi}

\author[J. Maynard]{James Maynard}
\address{Mathematical Institute
University of Oxford
Andrew Wiles Building, Oxford
OX2 6GG, UK}
\email{james.alexander.maynard@gmail.com}

 \author[I. E. Shparlinski] {Igor E. Shparlinski}

\address{Department of Pure Mathematics, University of New South Wales,
Sydney, NSW 2052, Australia}
\email{igor.shparlinski@unsw.edu.au}

\begin{abstract}  
We prove that there exist positive  constants $C$ and $c$ such that for any integer $d \ge 2$  the set  of $\vx\in [0,1)^d$ satisfying
$$
cN^{1/2}\le \left|\sum_{n=1}^N \exp\(2 \pi i\(x_1n+\ldots+x_d n^d\)\) \right|\le C N^{1/2}
$$
 for infinitely many natural numbers $N$ is of full Lebesque measure. 
This substantially improves the previous results where similar sets have been measured in terms of the Hausdorff dimension. 
We also obtain 
similar bounds for  exponential sums  with monomials $xn^d$ when $d\neq 4$. 
Finally, we obtain lower bounds for the Hausdorff dimension of large values of general exponential  
polynomials. 
\end{abstract}

%We prove that for any integer $d \ge 2$ there is a constant $c(d)> 0$,  which depends only on $d$, such that the set of  $\vx\in [0,1)^d$ with
%$$
%\left|\sum_{n=1}^N \exp\(2 \pi i\(x_1n+\ldots+x_d n^d\)\) \right|\ge c(d) N^{1/2}
%$$
% for infinitely many natural numbers $N$ is of full Lebesque measure. 
%This substantially improves the previous results where similar sets have been measured in term of the Hausdorff dimension. 
%We also obtain 
%similar bounds for   the set of large sums with monomials $xn^d$. 

\keywords{Weyl sums, Hausdorff dimension}
\subjclass[2010]{11L15, 28A78}

\maketitle

\tableofcontents

\section{Introduction}

\subsection{Background and motivation}

For an integer $d \geqslant 1$, let  
$$
\Tor =  (\R/\Z)^d
$$
be  the  $d$-dimensional unit torus. 
When $d=1$ we write
$$
\T = \T_1 = \R/\Z.
$$
For  a vector $\vx = (x_1, \ldots, x_d)\in \Tor$ and integer $N,$ we consider the exponential   
sums
$$
S_d(\vx; N)=\sum_{n=1}^{N}\e\(x_1 n+\ldots +x_d n^{d} \), 
$$
which  are commonly called {\it Weyl sums\/}, where   throughout  the paper we denote $\e(x)=\exp(2 \pi i x)$.  These sums were originally introduced by Weyl to study equidistribution of fractional parts of polynomials and rose to prominence through applications to the circle method and Riemann zeta function. Despite more than a century since these sums were introduced, their behaviour for individual values of $\vx$ is not well understood, see~\cite{Brud,BD}.
 
Much more is known about the average behaviour of $S_d(\vx;N)$. The recent advances of  Bourgain, Demeter and Guth~\cite{BDG} (for $d \geqslant 4$) 
and Wooley~\cite{Wool2} (for $d=3$)  (see also~\cite{Wool5}) for  the Vinogradov mean value theorem imply the estimate
\begin{equation}
\label{eq:MVT}
 N^{s(d)} \le \int_{\Tor} |S_d(\vx; N)|^{2s(d)}d\vx \leqslant  N^{s(d)+o(1)}, 
\end{equation} 
where   
$$
s(d)=\frac{d(d+1)}{2}
$$
and is best possible up to $o(1)$ in the exponent of $N$.

We observe that the optimal bound~\eqref{eq:MVT} does not tell much about the 
typical size of sums $S_d(\vx; N)$. It is conceivable, however unlikely, that the average value is influenced by  a very small set of $\vx \in \Tor$,
while for other 
$\vx \in \Tor$  these sums are very small. 
The main goal of this paper is to rule out this possibility and show that for
almost all  $\vx \in \Tor$ the sums $S_d(\vx; N)$ have order corresponding to the average size $N^{1/2}$
for infinitely many $N$.

\subsection{Previous results and questions}

The first results concerning the metric behaviour of Weyl sums are due to Hardy and Littlewood~\cite{HL1} who have estimated  the  {\it Gauss sums\/}  
\begin{equation}
\label{eq:GaussSum}
G(x; N) = \sum_{n=1}^N \e\(x n^2\),
\end{equation}
in terms of the continued fraction expansion of $x$. This idea has  been expanded upon by Fiedler, Jurkat and K\"orner~\cite[Theorem~2]{FJK} who give  the following optimal lower and upper bounds.
 Suppose that  $\{f(n)\}_{n=1}^{\infty}$ is a non-decreasing sequence of positive numbers. Then one has 
\begin{equation}
\begin{split}
\label{eq:GS}
\uplim_{N\to  \infty} \frac{ \left |G(x; N)\right |}{\sqrt{N} f(N)}<\infty & \text{ for almost all  }x\in \T\\
&  \qquad  \Longleftrightarrow \quad \sum_{n=1}^{\infty} \frac{1}{n f(n)^{4}} <\infty.
\end{split}
\end{equation}
See also~\cite[Theorem~0.1]{FK} for similar results with the more general sums $S_2(\vx;N)$,  (which correspond to $G(x;N)$ with a linear term in the phase).

For $d\ge 3$,  it has been  shown that for almost all $\vx\in \T_d$ 
$$
|S_{d}(\vx; N)|\le N^{1/2+o(1)}, \qquad N\rightarrow \infty.
$$
It has also been conjectured that the exponent $1/2$ is best possible, see~\cite[Conjecture~1.1]{ChSh-AM}. 
In this paper, among other things we confirm this conjecture, see Theorem~\ref{thm:gen sum} below.

For $\alpha\in (0, 1)$ and integer $d\ge 2$, consider the set
$$
\sE_{d,\alpha}=\{\vx\in \Tor:~|S_{d}( \vx; N)|\ge N^{\alpha} \text{ for infinity many }N\in \N\}.
$$
We remark that in the series of works~\cite{ChSh-AM, ChSh-IMRN, ChSh-JNT} for any $\alpha\in (0, 1)$ some upper and lower bounds have been given on the Hausdorff dimension $\dim \sE_{d, \alpha}$ of 
$\sE_{ d,\alpha}$ (see Definition~\ref{def:HD} below).

In Section~\ref{sec:heuristic}  we present some heuristic arguments about the exact behaviour of  $\dim \sE_{d, \alpha}$ for $\alpha\in (1/2, 1)$.

Furthermore, as in~\cite{ChSh-JNT}, we  also investigate   Weyl sums with  monomials
\begin{equation}
\label{eq:sigmadef}
\sigma_d(x; N)=\sum_{n=1}^N \e\(xn^d\).
\end{equation}
For each $\alpha\in (0, 1)$ let 
\begin{equation}
\label{eq:sigmaFdef}
\sF_{d, \alpha}=\{x\in \TTT:~|\sigma_d(x; N)|\ge N^{\alpha} \text{ for infinitely many } N\in \N\}.
\end{equation}   
Similarly to $\sE_{d, \alpha}$, for $\alpha\in (0,1)$ and integer $d\ge 2$ the set $\sF_{d, \alpha}$ has positive Hausdorff dimension. 
Moreover for $\alpha\in (1/2, 1)$ and $d\ge 2$ the set $\sF_{d, \alpha}$ has zero Lebesgue measure~\cite[Corollary~2.2]{ChSh-IMRN}.

Our method also shows that for $\alpha =1/2$ a slight modification of the sets
$\sF_{d, 1/2}$ (with $d=3$ or $d \ge 5$) and $ \cE_{d, 1/2}$ (with any $d \ge 3$), see~\eqref{eq:Ecd-mon} and~\eqref{eq:Ecd-gen} below,
are of full Lebesgue 
 measure. This implies that  
\begin{equation}
\label{eq:Hd=1}
\dim \sF_{d, \alpha}=1 \mand \dim \cE_{d, \alpha} = d,\qquad  \forall \, \alpha\in (0,1/2).
\end{equation}  
We remark that~\eqref{eq:Hd=1}  also applies to $d=4$. In Theorem~\ref{thm:mon sum 4} below 
 we only establish the positivity of the Lebesgue measure for $d=4$. This nevertheless is still enough to 
conclude that $\dim \sF_{4, \alpha}=1$ for all $\alpha\in (0,1/2)$.

\subsection{Notation and conventions}
Throughout the paper, the notation $U = O(V)$, 
$U \ll V$ and $ V\gg U$  are equivalent to $|U|\leqslant c V$ for some positive constant $c$, 
which depends on the degree $d$ and occasionally on the small real positive 
parameter $\varepsilon$.  We never {\it explicitly\/} mention these dependences, 
but we do this for other parameters such as the function $f$, the interval $\fI$ and the cube
$\fQ$. 

 We also define $U \asymp V$ as an equivalent $U \ll V \ll U$. 

For any quantity $V> 1$ we write $U = V^{o(1)}$ (as $V \to \infty$) to indicate a function of $V$ which 
satisfies $ V^{-\eps} \le |U| \le V^\eps$ for any $\eps> 0$, provided $V$ is large enough. One additional advantage 
of using $V^{o(1)}$ is that it absorbs $\log V$ and other similar quantities without changing  the whole 
expression.  

%% We use $\# \cS$ to denote the cardinality of a finite set $\cS$. 
For a a finite set $\cS$, we use $\# \cS$ to denote its cardinality.

 We always identify $\Tor$ with half-open unit cube $[0, 1)^d$.  

We say that some property holds for almost all $\vx \in \T_k$ if it holds for a set 
 $\cX \subseteq [0,1)^k$ of $k$-dimensional  Lebesgue measure  $\lambda(\cX) = 1$. 
 
When there is no confusion of positivity of $n$, we also  use $\sum_{n\le N} a_n$ to represent the sum $\sum_{n=1}^N a_n$.

\section{Main results}
\label{sec:main}

\subsection{Results on the Lebesgue measure} 
\label{sec:Lebesgue}
It is convenient to introduce a weighted variant  of the sums $\sigma_d(x; N)$ and $S_d(\vx; N)$.
In particular, for a sequence of complex weights $\va=(a_n)_{n=1}^\infty$  with 
$|a_n|=1$ we define 
\begin{align*}
& \sigma_{\va, d}(x; N)=\sum_{n=1}^N a_n \e\(xn^d\), \\ 
& S_{\va, d}(\vx; N)=\sum_{n=1}^Na_n \e\(x_1n+\ldots + x_dn^d\).
\end{align*}  
 
Here we are mostly interested in the case $\alpha=1/2$. Hence
we  modify the notations  for  $\sE_{d,  1/2}$ and $\sF_{d, 1/2}$ in a way that they also apply to $S_{\va, d}(\vx; N)$ and $\sigma_{\va, d}(x; N)$. 

For integer $d\ge 2$ and  constants $c, C>0$ denote 
\begin{equation}
\label{eq:Ecd-gen}
\begin{split}
\cE_{\va, c, C}(d)=\{\vx\in \Tor&:~ cN^{1/2} \le |S_{\va, d}(\vx; N)|\le CN^{1/2} \\
& \qquad \quad  \text{ for infinitely many }   N\in \N\},
\end{split}
\end{equation}
and
\begin{equation}
\label{eq:Ecd-mon}
\begin{split}
\cF_{\va, c, C}(d)=\{x\in \TTT &:~ cN^{1/2}\leq  |\sigma_{\va, d}(x; N)|\le C N^{1/2}\\
& \qquad \quad \text{ for infinitely many } N\in \N\}.
\end{split}
\end{equation}

For more general sets $\cA\subseteq \Tor$ we use $\lambda(\cA)$ to denote the Lebesgue measure of $\cA$.

We start with the case of monomial sums.

\begin{theorem}
\label{thm:mon sum}
There exist positive constants  $c$,  $C$ such that for  $d=3$ or $d\ge 5$ and 
 any sequence of complex weights $\va=(a_n)_{n=1}^\infty$  with 
$|a_n|=1$ we have    $\lambda(\cF_{\va, c, C}(d)) =1$. 
\end{theorem}

Note that  there are still  exceptional values $d=2$ and $d=4$ to which Theorem~\ref{thm:mon sum} does not apply. 
For $d=4$ we however are still able to show that the set $\cF_{\va, c}(d)$ is  everywhere massive, where 
\begin{align*}
\cF_{\va, c}(d)=\{x\in \TTT &:~|\sigma_{\va, d}(x; N)|\ge c N^{1/2}\\
& \qquad \quad \text{ for infinitely many } N\in \N\}.
\end{align*}
See Remark~\ref{rem:Hooley} for a possible approach to extending Theorem \ref{thm:mon sum} to cover the case $d=4$.

\begin{theorem}
\label{thm:mon sum 4}
Let  $0<c<1$. Then for any sequence of complex weights $\va=(a_n)_{n=1}^\infty$  with 
$|a_n|=1$,  and for any  interval $\fI \subseteq \TTT$  we have 
$$
\lambda \(\cF_{\va, c}(4) \cap \fI\)\ge  (\lambda(\fI)(1-c^{2}))^{2}/8.
$$ 
\end{theorem}

We remark that for fixed  constants $ C>c > 0$ our method does not yield that $\lambda \(\cF_{\va, c, C}(4) \cap \fI\)>0$ for every interval $\fI\subseteq \TTT$. Unfortunately the conclusion of Theorem~\ref{thm:mon sum 4}, that is 
$$
\lambda \(\cF_{\va, c}(4) \cap \fI\)\gg \lambda(\fI)^2, \qquad \forall\, \fI\subseteq \TTT,
$$
does not imply $\lambda(\cF_{\va, c}(4))=1$. Indeed, consider the set of $\cG_n$ of fractions $a/3^n$ with $1\le a\le 3^n$. 
We now define
$$
\cA =\TTT \cap \bigcup_{n\in \N}  \bigcup_{a/q \in \cG_n} [a/q-    3^{-n-2} n^{-2}, a/q + 3^{-n-2} n^{-2}]. 
$$
Clearly 
$$
\lambda (\cA) \le \frac{2}{9}   \sum_{n=1}^\infty  n^{-2}  =  \frac{2 \pi^2}{54} < 1.
$$
On the other hand, for each interval $\cJ = [x_0, x_0+\delta] \subseteq \TTT$ with $0<\delta<\delta_0$ for some small $\delta_0$  there is  an integer $n$ such that 
\begin{equation}
\label{eq:n-delta}
3^{-n}n^{3}\le \delta< 3^{-n+1}(n-1)^{3},
\end{equation}
and there are at least 
$ n^{3}/3$ fractions
$$
 a/3^n \in \cG_n \cap \cJ.
 $$
Hence combining with~\eqref{eq:n-delta} we obtain
$$
\lambda \(\cA \cap \cJ\)  \gg 3^{-n} n\gg  \delta  \(\log \delta^{-1}\)^{-2}.
$$
It is easy to see that one can modify this construction to replace $3^n$ with a faster growing function  and  $  \(\log \delta^{-1}\)^{-2}$
with  a  slower decaying function, in fact with an arbitrary slow rate of decay.

We now turn to the Weyl sums $S_d(\vx; N)$. First observe that for $\vx=(x_1, \ldots, x_d)$ we have 
$$
S_{\va, d}(\vx; N)=\sigma_{\vb, d}(x_d; N),
$$
where 
$$
b_n=a_n\e(x_1n+\ldots+x_{d-1}n^{d-1}).
$$
Thus for $d=3$ or $d\ge 5$ and  any fixed $(x_1, \ldots, x_{d-1})\in \T_{d-1}$, Theorem~\ref{thm:mon sum} implies $\lambda(\cF_{\vb, c, C}(d))=1$.  Together with Fubini's theorem we obtain  $\lambda(\cE_{\va, c, C}(d))=1$. By introducing a new idea we obtain the following  desired result for all $d\ge 2$.

\begin{theorem}
\label{thm:gen sum}
There exist positive constants  $c$,  $C$ such that for all $d\ge 2$ and 
 any sequence of complex weights $\va=(a_n)_{n=1}^\infty$  with 
$|a_n|=1$ we have    $\lambda(\cE_{\va, c, C}(d)) =1$.
\end{theorem}

We remark that~\cite[Theorem~0.1]{FK} gives an optimal bound for the sums  
$S_2(\vx; N)$.  However, for sums with weights, Theorem~\ref{thm:gen sum} is new even for $d=2$.

It is interesting to understand whether  the constant $c$ of Theorem~\ref{thm:mon sum}  can be any arbitrary large (also whether the cases of $d=2, 4$ can be included in 
Theorem~\ref{thm:mon sum}).  More precisely we ask the following.

\begin{question}
\label{quest:large sigma d}
Let $d\ge 2$ and   $\va=(a_n)_{n=1}^\infty$ a sequence of complex weights  with 
$|a_n|=1$. Is this true that for almost all $x\in \TTT $  we have 
$$
\limsup_{N\rightarrow \infty} \frac{ \sigma_{\va, d}(x; N)}{\sqrt{N}}=\infty?
$$ 
\end{question}

We note for $d=1$ the answer to  Question~\ref{quest:large sigma d}, that is, for standard  trigonometric polynomials,
 is negative as by an explicit construction of Hardy and Littlewood~\cite[Section~4]{HL2} which states that for 
any $\xi \in \R$ with $\xi\ne 0$ we have
$$
\sup_{x \in \TTT } \left | \sum_{n=1}^N \e\(\xi n \log n + xn\)\right|  \ll N^{1/2}, 
%% \ll_\xi N^{1/2}, 
$$
(where the implied constant may depend on $\xi$), 
see also a result of Rudin~\cite[Theorem~1]{Rud} who has shown the same ``flatness'' can be achieved
for partial sums trigonometric series with coefficients $a_n = \pm 1$.

%%We believe that the constant $c$ of Theorem~\ref{thm:mon sum}  can be any arbitrary large
%% (also that the case of $d=2, 4$ can be included in 
%%Theorem~\ref{thm:mon sum}).  More precisely we propose the following.
%%
%%\begin{conj}
%%Let $d\ge 2$ and   $\va=(a_n)_{n=1}^\infty$ a sequence of complex weights  with 
%%$|a_n|=1$. Then for almost all $x\in \TTT $  we have 
%%$$
%%\limsup_{N\rightarrow \infty} \frac{ \sigma_{\va, d}(x; N)}{\sqrt{N}}=\infty.
%%$$ 
%%\end{conj}

\subsection{Results on the Hausdorff dimension} 
\label{sec:hausdroff}  
For Gauss sums~\eqref{eq:GaussSum} we have an optimal result in~\eqref{eq:GS}. 
However,  the Diophantine approximation argument of~\cite[Theorem~2]{FJK} does not work  for  Gauss sums with weights. Moreover, our method does not give positive measure for $\cF_{\va, c}(2)$ either. However, by introducing some new ideas we  obtain a lower bound of the Hausdorff dimension of the set $\cF_{\va, c}(2)$. Indeed our method works for more general functions  $f$. 

\begin{definition}
\label{def:HD}
The  Hausdorff dimension of a set $\cA\subseteq \R^{d}$ is defined as 
\begin{align*}
\dim \cA=\inf\Bigl\{s>0:~\forall \, & \eps>0,~\exists \, \{ \cU_i \}_{i=1}^{\infty}, \ \cU_i \subseteq \R^{d},\\
&  \text{such that } \cA\subseteq \bigcup_{i=1}^{\infty} \cU_i \text{ and } \sum_{i=1}^{\infty}\(\diam\cU_i\)^{s}<\eps \Bigr\}.
\end{align*}
\end{definition}
We refer  to~\cite{Falconer, Mattila1995} for a background on the Hausdorff dimension.

\begin{theorem}
\label{thm:f}
Let $f$ be a real, twice  differentiable function with continuous second derivative satisfying 
$$
  f''(t)=t^{\gamma-2+o(1)}
$$ 
for some $\gamma> 2$. Then for any interval $\fI\subseteq \R$ the Hausdorff dimension of the set of $x\in \fI$ such that
$$
\left|\sum_{1\le n \le N}a_n\e(xf(n))\right| \gg N^{1/2} \quad \text{for infinitely many $N$,}
$$
where the implied constant may depend  on the function $f$, 
is at least $1-1/(2\gamma)$.
\end{theorem}

If we impose conditions only on the first derivative of  the function $f$ in Theorem~\ref{thm:d=2}  we obtain the following weaker bound.

\begin{theorem}
\label{thm:d=2} Let $f$ be a real, continuously differentiable function such that 
$$
f'(t)= t^{\gamma-1+o(1)} 
$$ 
for some $\gamma>1$. Then for any complex weights $\va=(a_n)_{n=1}^\infty$  with 
$|a_n|=1$ and any interval $\fI \subseteq \R$ the Hausdorff dimension of the set of $x\in \fI$ such that
$$
\left|\sum_{1\le n \le N}a_n\e(xf(n))\right| \gg N^{1/2} \quad \text{for infinitely many $N$,}
$$
where the implied constant may depend on the function $f$, 
is at least $1-1/\gamma$.
\end{theorem}

Theorems~\ref{thm:f} and~\ref{thm:d=2}  are based on some results on the distribution of values of 
exponential polynomials,  which we develop in Section~\ref{sec:ExpPoly}.

\subsection{Applications to uniform distribution modulo one} 
%\label{sec:u.d.}  

We now  show some applications of our main results to the theory of uniform distribution of sequences.  

Let $\xi_n$, $n\in \N$,  be a sequence in $\TTT$. The {\it discrepancy\/}  of this sequence at length $N$ is defined as 
\begin{equation}
\label{eq:Discr}
D_N = \sup_{0\le a<b\le 1} \left |  \#\{1\le n\le N:~\xi_n\in (a, b)\} -(b-a) N \right |.
\end{equation} 

Recall that a sequence is uniformly distributed modulo one if and only if the corresponding discrepancy satisfies
$$
D_N=o(N) \qquad \text{as} \ N \to \infty,
$$
see~\cite[Theorem~1.6]{DrTi} for a proof.  We note that sometimes in the literature the scaled quantity $N^{-1}D_N $ is   called the discrepancy, but 
since our argument looks cleaner with the definition~\eqref{eq:Discr}, we adopt it here. 

For $\vx\in \Tor$ and the sequence 
$$
\xi_n=x_1n+\ldots +x_d n^{d}, \qquad n\in \N,
$$
we  denote by $D_d(\vx; N)$ the corresponding discrepancy.  Motivated by the work of Wooley~\cite[Theorem~1.4]{Wool3}, 
it has been  shown in~\cite{ChSh-IMRN}  
that for almost all $\vx\in \Tor$ with $d\ge 2$ one has 
$$
D_d(\vx; N)\le N^{1/2+o(1)} \qquad \text{as} \ N \to \infty.
$$

Recalling the  {\it Koksma-Hlawlka inequality\/}, see~\cite[Theorem~1.14]{DrTi} for a general statement, we  derive 
for any $\vx\in \Tor$
$$
S_d(\vx; N)\ll D_d(\vx; N).
$$
Combining with Theorem~\ref{thm:gen sum} we conclude that there is a constant $c>0$ such that  for almost all $\vx\in \T_d$,
$$
D_d(\vx; N)\ge c N^{1/2}
$$
holds for infinitely many $N\in \N$. 

Similarly,   other results  from  Section~\ref{sec:main}   lead to lower bounds of the discrepancy of the corresponding sequences.

\section{Preliminaries}

\subsection{Reduction to power moments} 
We first show how our results 
of Section~\ref{sec:Lebesgue} can be reduced to estimating  the second and fourth moment of exponential sums. Our first result is a variation of Cassels~\cite[Lemma~1]{Cas}.

\begin{lemma}
\label{lem:measure}
Let $\cX\subseteq \T_d$  be measurable with $\lambda(\cX)>0$. Let $f:\T_d\rightarrow[0,N]$ be a continuous function. Suppose that there are positive constants $\alpha_1, \alpha_2$ such that 
\begin{equation}
\label{eq:2}
\int_{\cX} f(\vx)^{2}d\vx \ge \alpha_1 N \lambda(\cX)
\end{equation}
and 
\begin{equation}
\label{eq:4}
\int_{\cX} f(\vx)^{4} d\vx \le \alpha_2 N^2 \lambda(\cX).
\end{equation}
Then for any constants $c, C>0$ we have 
$$
\lambda\left(\left\{\vx\in \cX:~cN^{1/2}\le f(\vx)\le  CN^{1/2}\right\} \right ) \ge \varepsilon_0\lambda(\cX),  
$$
where 
$$
\varepsilon_0=(\alpha_1-c^2-\alpha_2/C^2)/C^2.
$$
\end{lemma}

\begin{proof}
Denote 
\begin{align*}
&\cA_{c}=\left\{\vx\in \cX:~f(\vx)< c N^{1/2}\right\},\\
&\cB_{C}=\left\{\vx\in \cX:~f(\vx)> C N^{1/2}\right \}, 
\end{align*}
and 
$$
\cR_{c, C}= \cX \setminus \(\cA_{c}\cup \cB_{C}\).
$$

Since $f$ is continuous, the sets $\cA_c,\cB_C,\cR_{c,C}$ are measurable. We note 
that~\eqref{eq:4} implies 
$$
\int_{\cB_{C}} f(\vx)^2  d\vx \le \frac{1}{C^2N} \int_{\cX} f(\vx)^{4}  d\vx
\le    \alpha_2 N\lambda(\cX)/C^2.
$$
Taking a decomposition of  $\cX$ as $\cX=\cA_c\cup \cB_C\cup \cR_{c, C}$, we obtain 
$$
\int_{\cX} f(x)^{2} d\vx \le c^{2} N \lambda(\cX)+ \alpha_2 N\lambda(\cX)/C^2+\int_{\cR_{c, C}} f(\vx)^{2}d\vx.
$$
Combining with~\eqref{eq:2} and using that $f(\vx)\le CN^{1/2}$ whenever $\vx \in R_{c, C}$ gives
$$
\lambda(\cR_{C, c})\ge \lambda(\cX) \(\alpha_1-c^2-\alpha_2/C^2\)/C^2,
$$
which finishes the proof. 
\end{proof}

\begin{remark}
\label{rem:Lp}
We remark that the  bound~\eqref{eq:4} on the $L^4$-norm appears naturally in our argument. 
However, suppose that  for some $r>2$ we have the following bound on the  $L^r$-norm 
$$
\int_{\cX} f(\vx)^{r} d\vx \le \alpha_2 N^{r/2} \lambda(\cX). 
$$
Then we obtain the desired result of Lemma~\ref{lem:measure} as well.
\end{remark}

\begin{cor}
\label{cor:measure} Let $\cE_{\va, c, C}(d)$ be given by~\eqref{eq:Ecd-gen}. Suppose that for each cube $\fQ\subseteq \T_d$ and each integer $N$ which is sufficiently large (in terms of $\fQ$) we have 
\begin{equation} 
\begin{split}
\label{eq:S24}
&\int_{\fQ}|S_{\va, d}(\vx; N)|^2 d\vx \ge  \alpha_1\lambda(\fQ)N, \\
&\int_{\fQ}|S_{\va, d}(\vx; N)|^4d\vx \le \alpha_2\lambda(\fQ)N^2.
\end{split}
\end{equation} 
Then 
$$\lambda(\cE_{\va, c, C}(d)\cap \fQ)\ge \varepsilon_0  \lambda(\fQ),$$   
where 
$$
\varepsilon_0=(\alpha_1-c^2-\alpha_2/C^2)/C^2.
$$
\end{cor}

 \begin{proof}
Define
$$
\cL_{N,c, C}=\left\{ \vx\in \fQ :~cN^{1/2}\le |S_{\va, d}(\vx; N)|\le CN^{1/2}\right\}
$$
so that 
$$
\cE_{\va, c, C}(d)=\bigcap_{M \ge 1} \bigcup_{N\ge M} \cL_{N, c,C}.
$$
From Lemma~\ref{lem:measure} and~\eqref{eq:S24}, for each $N\ge 1$ we have 
$$\lambda(\cL_{N,c, C})\ge \varepsilon \lambda(\fQ).$$ 
Hence by continuity of Lebesgue measure, see  for example~\cite[Theorem~1.4,~(4)~(ii)]{Mattila1995}, we get
$$
\lambda \left (\bigcap_{M \ge 1} \bigcup_{N\ge M} \cL_{N,c, C}\right )=\lim_{M\rightarrow \infty} \lambda\(\bigcup_{N\ge M} \cL_{N,c, C} \) \ge \varepsilon_0\lambda(\fQ),
$$
which completes the proof.
\end{proof}

The following is a variant of a result due to Cassels~\cite{Cas}, see also~\cite[Lemma~2]{Gall}.

\begin{lemma}
\label{lem:cass}
Let $\fQ_k\subseteq \R^d$ be a sequence of cubes and $\fU_k\subseteq \R^{d}$ a sequence of Lebesgue measurable sets, $k =1,2\ldots$,  such that for some positive $\varepsilon<1$ 
$$
\fU_k\subseteq \fI_k, \quad \lambda(\fU_k)\ge \varepsilon \lambda(\fQ_k), \quad \lambda(\fQ_k)\rightarrow 0.
$$
Then the set of points which belong to infinitley many $\fQ_k$ has the same measure as the set of points which belong to infinitley many of the $\fU_k$.
\end{lemma}

Combining Corollary~\ref{cor:measure} with Lemma~\ref{lem:cass}, we show that the equality $\lambda(\cE_{\va, c, C}(d))=1$ follows from moment estimates for Weyl sums. Note that we could also derive this conclusion from Corollary~\ref{cor:measure}  and the Lebesgue density theorem~\cite[Corollary~2.14]{Mattila1995}.

\begin{lemma}
\label{cor:main}
Suppose that for each cube $\fQ\subseteq \T_d$ and each integer $N$ which is sufficiently large (in terms of $\fQ$) we have 
\begin{equation}
\label{eq:S24-1}
\int_{\fQ}|S_{\va, d}(\vx; N)|^2d\vx\gg \lambda(\fQ)N, \quad \int_{\fQ}|S_{\va, d}(\vx; N)|^4d\vx\ll \lambda(\fQ)N^2.
\end{equation}
Then there are positive constants $c$,  $C$ such that  $\lambda(\cE_{\va, c, C}(d))=1$.
\end{lemma}

\begin{proof}
This follows by applying Lemma~\ref{lem:cass} to a sequence of cubes with diameter tending to zero centered at points from a countable dense subset of $\T_d$ and using~\eqref{eq:S24-1} and Corollary~\ref{cor:measure} to verify the conditions of Lemma~\ref{lem:cass} are satisfied.
\end{proof}

We {\it emphasise\/} that the implied constant in~\eqref{eq:S24-1} can only depend on the ambient dimension $d$ and cannot depend on $\fQ$.

A similar argument allows us to deal with monomials.
\begin{lemma}
\label{cor:main1}
Suppose that for each interval $\fI\subseteq \T$ and each integer $N$ which is sufficiently large (in terms of $\fI$) we have 
\begin{equation}
\label{eq:S24-mon}
\int_{\fI}|\sigma_{\va, d}(x; N)|^2dx\gg \lambda(\fI)N, \quad \int_{\fI}|\sigma_{\va, d}(x; N)|^4dx\ll \lambda(\fI)N^2.
\end{equation}
Then  there are positive constants $c$,  $C$ such that $\lambda(\cF_{\va, c, C}(d))=1$.
\end{lemma}
 
In order to prove Theorems~\ref{thm:mon sum} and~\ref{thm:gen sum} it is sufficient to establish~\eqref{eq:S24-1} and~\eqref{eq:S24-mon}.  
These results are presented Sections~\ref{sec:sumdiff pow},  \ref{sec:2nd mom} and~\ref{sec:fourth}.

Note that the {\it Rudin conjecture\/}~\cite[Conjecture~3]{CG} asserts that for any $2<r<4$ and any complex sequence $a_n$ we have 
\begin{equation}
\label{eq:Rudin}
\int_{\T}\left |\sum_{n=1}^{N}a_n \e(xn^{2}) \right |^{r}dx\ll \left (\sum_{n=1}^{N}|a_n|^{2}\right )^{r/2}, 
\end{equation}
where the implied constant may depend on $r$. 
Combining~\eqref{eq:Rudin} with Lemma~\ref{lem:measure} and Remark~\ref{rem:Lp} we conclude that the Rudin conjecture implies that 
there are positive constants $c$,  $C$ such that $\lambda(\cF_{\va, c, C}(2))>0$ (under the condition $|a_n|=1$). 
Furthermore, suppose that there is some $r>2$ such that for any interval $\fI\subseteq \T $ and any complex sequence $a_n$ with $|a_n|=1$ we have (the local version of the Rudin conjecture)
$$
\int_{\fI}\left |\sum_{n=1}^{N}a_n \e(xn^{2}) \right |^{r}dx\ll  N^{r/2} \lambda(\fI),
$$
provided that $N$ is sufficiently large in terms of $\fI$, 
then combining with  Lemma~\ref{lem:measure}, Remark~\ref{rem:Lp} and Lemma~\ref{lem:cass} there are  positive constants $c, C$ such that 
$\lambda(\cF_{\va, c, C}(2))=1$.  However, the Rudin conjecture does not answer the Question~\ref{quest:large sigma d} for the case $d=2$.

%We remark that we may obtain the claim  $\lambda(\cF_{\va, c, C}(2))=1$ by using some other simpler arguments other than proving \commB{I'm not so sure I follow this remark, are we able to show this with simpler arguments?}the Rudin's conjecture. 
%Moreover, Note that even under the assumption of the Rudin conjecture, we do not answer the Question~\ref{quest:large sigma d} for the case $d=2$.

\subsection{Some tools from harmonic analysis}
%% \label{sec:prelim}

We need the following obvious  identity.

 \begin{lemma}
\label{lem:diag-2nu}   Let  $0 < \delta\le 1$, 
  $y_1,\ldots, y_K$ be a sequence of real numbers and $\beta_1,\ldots, \beta_K $ be a sequence of complex numbers. For any integer $\nu \ge 1$, we have 
$$
\int_0^\delta  \left|\sum_{k=1}^{K}\beta_k\e\(zy_k\)\right|^{2\nu} dz = \sfM + \sfE, 
$$
where 
\begin{align*}
\sfM&=\delta  \sum_{\substack{ 1\le k_1, \ldots , k_\nu, \ell_1, \ldots ,  \ell_\nu\le K\\ y_{k_1}+\ldots + y_{k_\nu} = y_{\ell_1} +\ldots + y_{\ell_\nu}}} \beta_{k_1}\ldots \beta_{k_\nu}
 \overline {\beta_{\ell_1}} \ldots  \overline {\beta_{\ell_\nu}}, \\
\sfE&=\sum_{\substack{ 1\le k_1, \ldots , k_\nu, \ell_1, \ldots ,  \ell_\nu\le K\\ y_{k_1}+\ldots + y_{k_\nu} \ne y_{\ell_1} +\ldots + y_{\ell_\nu}}}
 \frac{ \beta_{k_1}\ldots \beta_{k_\nu}
 \overline {\beta_{\ell_1}} \ldots  \overline {\beta_{\ell_\nu}}}{2\pi i \( y_{k_1}+\ldots + y_{k_\nu} - y_{\ell_1} -\ldots - y_{\ell_\nu}\)}\\
& \qquad \qquad \qquad  \qquad \times \(\e\(\delta\( y_{k_1}+\ldots + y_{k_\nu} - y_{\ell_1} -\ldots - y_{\ell_\nu}\)\)-1\).
\end{align*}
\end{lemma} 
\begin{proof}
This follows after expanding the square, interchanging summation and evaluating the integral.
\end{proof}

The above result may be applied to obtain an asymptotic formula for various integrals. In some cases it is  technically convenient to work with smooth weights at the cost of establishing only upper and lower bounds. Results of this type are well known and we provide a typical proof.
\begin{lemma}
\label{lem:eqntomv}
Let $I$ be an interval and $\varphi_1,\varphi_2,\ldots \varphi_{k}$ real valued functions on $I$. For any $Y_1,\ldots,Y_k\gg 1$ and sequence of complex numbers $a_n$ satisfying $|a_n|\le 1$ we have 
\begin{align*}
&\frac{1}{Y_1\ldots Y_k}\int_{-Y_1}^{Y_1}\ldots \int_{-Y_k}^{Y_k}\left|\sum_{n\in I} a_n \e\left(\sum_{i=1}^{k}y_i\varphi_i(n) \right)\right|^{4}d y_1\ldots d y_k\\
& \quad \ll \# \left \{ n_1,\ldots,n_4\in I: ~\left|\varphi_i(n_1)+\cdots-\varphi_i(n_4)\right|\le \frac{1}{Y_i}, \   1\le i \le k \right\}.
\end{align*}
\end{lemma}

\begin{proof}
Let $F$ be a positive smooth function with sufficient decay satisfying  
$$
F(x)\gg 1 \ \ \text{if} \ \  |x|\le 1 \quad \text{and} \quad \text{supp} \, \widehat F\subseteq [-1,1].
$$
where $ \text{supp} \, \widehat F = \{x\in \R:~ \widehat F(x) \ne 0\}$.  
We have  
\begin{align*}
&\int_{-Y_1}^{Y_1}\ldots \int_{-Y_k}^{Y_k}\left|\sum_{n\in I} a_n \e\left(\sum_{i=1}^{k}y_i\varphi_i(n) \right)\right|^{4}d y_1\ldots d y_k \\ &
\qquad  \ll  \int_{-\infty}^{\infty}\ldots \int_{-\infty}^{\infty} \prod_{i=1}^kF(y_i/Y_i) \left|\sum_{n\in I} a_n \e\left(\sum_{i=1}^{k}y_i\varphi_i(n) \right)\right|^{4}d y_1\ldots d y_k. 
\end{align*}
Expanding the fourth power, interchanging summation, recalling the assumption $|a_n|\le1$  and using Fourier inversion gives 
\begin{align*}
&\int_{-Y_1}^{Y_1}\ldots \int_{-Y_k}^{Y_k} \left|\sum_{n\in I} a_n \e\left(\sum_{i=1}^{k}y_i\varphi_i(n) \right)\right|^{4}d y_1\ldots d y_k \\ & \quad \quad \quad \quad \ll Y_1\ldots Y_k\sum_{n_1,\ldots,n_4\in I}\prod_{i=1}^{k}\left|\widehat F\left(Y_i(\varphi_i(n_1)+\ldots-\varphi_i(n_4)) \right)\right|,
\end{align*} 
and the result follows from  $\text{supp}\, \widehat F\subseteq[-1,1]$.  
\end{proof}

 \subsection{Number of representations by sums and differences of powers}
\label{sec:sumdiff pow}
We next collect some   results on the number of representations $R_d(k,N)$  of 
an integer $k$ as
$$
k = n_1^d +  n_2^d - n_3^d  - n_4^d, \qquad 1\le n_1, n_2, n_3, n_4\le N.
$$ 
They are crucial for our bounds on moments of  exponential polynomials.

We first recall a result of Skinner and Wooley~\cite[Theorem~1.2]{SkinWool},
which treats the case of $k=0$ and shows that essentially all solutions are diagonal (that is, 
with $\{n_1, n_2\}= \{n_3,n_4\}$).   

\begin{lemma}
\label{lem:SkinWool}
For  $d \ge 2$ we have 
$$
R_d(0,N) = 2N^2 + O\( N^{3/2 + 1/(d-1) + o(1)}\).
$$
Moreover, when $d = 3$ or $d = 5$, one may replace the term $1/(d -1)$ in each of the above estimates by $1/d$. 
\end{lemma}

We note that~\cite[Theorem~1.2]{SkinWool} improves a series of previous results with weaker error terms,
each of them would be suitable for our purpose. On the other hand, one can improve~\cite[Theorem~1.2]{SkinWool} 
by using a result of Hooley~\cite[Theorem~3]{Hool2}, which however gives us no advantage: for several even stronger
bounds, see~\cite{Brow1, Brow2, BrH-B, H-B1, H-B2, Morm} and references therein.

For bounding $R_d(k,N)$ with $k\ne 0$ we need the  following result of  Marmon~\cite[Theorem~1.4]{Morm}.

\begin{lemma} 
\label{lem:Mormon}
Let $a_1, a_2, a_3, M$ be non-zero integers. Let $r(M, B)$ count the number of solutions $(x_1, x_2, x_3)\in \Z^{3}$ to the equation
$$
a_1x_1^{d}+a_2x^{d}+a_3x_3^{d}=M
$$
satisfying $|x_i|\le B$ and $a_ix_i^{d}\neq M$ for $i=1, 2, 3$. Then 
$$
%%r(M, B)=O_{d, \epsilon}(B^{2/d^{1/2}+o(1)}).
r(M, B)=O(B^{2/d^{1/2}+o(1)}).
$$
\end{lemma}
For $R_d(k, N)$ with $k\neq 0$ using Lemma~\ref{lem:Mormon} we obtain the following.

\begin{lemma}
\label{lem:Morm}
For  $d \ge 2$ and $k \ne 0$ we have 
$$
R_d(k,N)\le N^{1+2/d^{1/2}+o(1)} .  
$$
\end{lemma}

\begin{proof}  
We see that by Lemma~\ref{lem:Mormon} for any fixed $n_4$ there are at most 
$N^{2/d^{1/2}+o(1)}$ solutions to $ n_1^d +  n_2^d - n_3^d  = n_4^d+k$, 
$n_1, n_2, n_3 \le N$ unless 
\begin{equation}
\label{eq:n2n3}
n_1^d =  n_4^d+k, \qquad n_2 = n_3,
\end{equation}
 or 
\begin{equation}
\label{eq:n1n3}
n_2^d =  n_4^d+k, \qquad n_1 = n_3,
\end{equation} 
or
\begin{equation} 
\label{eq:k}
-n_3^d=n_4^d+k.
\end{equation} 
Thus the total contribution from such solutions (avoiding~\eqref{eq:n2n3}, \eqref{eq:n1n3}  and~\eqref{eq:k}) is at most $N^{1+2/d^{1/2}+o(1)}$.

Otherwise, using   the classical bound 
\begin{equation}
\label{eq:tau}
\tau(k) = k^{o(1)}, 
\end{equation} 
on the divisor function,  see~\cite[Equation~(1.81)]{IwKow}, we see that there are $k^{o(1)}$ pairs $(m,n)$ with 
$m^d =  n^d+k$ (which we write as  
$$
k= (m-n) (m^{d-1} + \ldots + n^{d-1}).
$$
Therefore, 
 the total contribution from the solution~\eqref{eq:n2n3} and~\eqref{eq:n1n3} is at most $N^{1+o(1)}$. Clearly there are at most $O(1)$ solutions to the equation~\eqref{eq:k} which leaves $O(N)$ solutions in remaining variables $n_1,n_2$.  Putting all this together we obtain the desired bound. 
\end{proof}

Lemma~\ref{lem:Morm} gives a satisfactory bound when $d \ge 5$. Unfortunately we do not have a good 
bound for $d\le 4$. However the classical argument of Hooley~\cite{Hool1} gives a suitable bound
for $d=3$.
 
 \begin{lemma}
\label{lem:Hooley}
For  $k \ne 0$ we have 
$$
R_3(k,N)\le N^{11/6+o(1)}.
$$
\end{lemma}

\begin{proof} We recall that  Hooley~\cite{Hool1} 
considers the equation $k = n_1^3 +  n_2^3 + n_3^3  + n_4^3$ 
with unrestricted positive integers $n_1, n_2, n_3, n_4$ from which of course 
follows that $n_1, n_2, n_3, n_4\le k^{1/3}$.  Thus, in our case $N$ replaces $k^{1/3}$
in the  argument of~\cite{Hool1}.

It is also important for~\cite{Hool1} that  the equation is fully symmetric and one can form a sum
$n_i^3 + n_j^3$ of two cubes of the same parity. 
Our equation $k = n_1^3 +  n_2^3 - n_3^3  - n_4^3$ lacks this symmetry, 
however we can instead consider the equation 
$$
8k = (2n_1)^3 + (2n_2)^3 - m_3^3 - m_4^4
$$
which has at least as many solutions, and after denoting $m_1=2 n_1$ and $m_2 = 2n_2$ we
regain the desired parity condition. 

One can verify that beyond these two points everything goes exactly as in~\cite{Hool1} 
and the sign changes do not affect the rest. Taking into account the range of variables $n_1,\ldots,n_4\le N$  we obtain the desired bound.  
\end{proof}

\section{Restriction bounds for moments of exponential sums}

\subsection{Second moments over small intervals and boxes}
\label{sec:2nd mom}

We now show that applying Lemma~\ref{lem:diag-2nu} to monomials of degree $d\ge 2$ gives an  
asymptotic formula for  integrals which are more general than  $I_{1,d}(\fI) $.

\begin{lemma}
\label{lem:aver f}
Let  $f$ be a real, continuously differentiable function such that 
$$
f'(t)=t^{\gamma-1+o(1)},  
$$ 
for some $\gamma>1$. Then  for any sequence of complex numbers  $\va=(a_n)_{n=1}^\infty$ with  $|a_n|=1$ and  any interval $\fI\subseteq   \R$ we have 
$$
\int_{\fI} \left|\sum_{n=N}^{2N}a_n\e\(xf(n)\)\right|^2d x  =  \lambda\(\fI\) N+O\(N^{2-\gamma+o(1)}\),
$$
where the implied constant depends on $f$.
\end{lemma}

\begin{proof} Suppose $\fI=[\alpha,\alpha+\delta]$.  
By  changing the coefficients
$a_n \rightarrow a_n \e(\alpha n^d)$ we may assume $\alpha = 0$.

 Using the assumption each $|a_n|=1$,  Lemma~\ref{lem:diag-2nu} implies  
\begin{align*}
\int_{\fI} &\left|\sum_{n=N}^{2N}a_n\e\(xf(n)\)\right|^2d x  \\
&\quad = \delta N +\sum_{\substack{N\le n_1,n_2\le 2N \\ n_1\neq n_2}}\frac{a_{n_1}\overline a_{n_2}\left(\e\(\delta\(f(n_1)-f(n_2)\)\)-1\right)}{2\pi i \(f(n_1)-f(n_2)\)} \\
&\quad = \delta N +O\(\sum_{N\le n_2<n_1\le 2N}\frac{1}{f(n_1)-f(n_2)}\)
\end{align*}  
(clearly we can assume that $N$ is large enough so $f(t)$ is monotonically increasing  for $n\ge N$). 
For any $N\le n_2<n_1\le 2N$, by the mean value theorem we have 
 $$
 f(n_1)-f(n_2)= (n_1-n_2) f'(\eta) \quad \text{for some $n_2\le \eta \le n_1$}.
 $$
Hence by assumption on $f'$
$$
 f(n_1)-f(n_2)\ge (n_1-n_2)N^{\gamma-1+o(1)}.
$$ 
Therefore, 
$$
\sum_{N\le n_2<n_1\le 2N}\frac{1}{f(n_1)-f(n_2)}\le N^{2-\gamma+o(1)},
$$
and the desired result follows.
\end{proof}

From Lemma~\ref{lem:aver f}, we immediately obtain an asymptotic formula for $I_{1,d}(\fI)$. Since
$$
 S_{\va, d}(\vx; N)=  \sigma_{\vb, d}(x_d; N), 
$$
 where 
 $$
 b_n = a_n \e\(x_1 n+\ldots +x_{d-1} n^{d-1} \),
$$
we may combine  Lemma~\ref{lem:aver f} with Fubini's theorem after covering the interval $[1,N]$ by $O(\log N)$ dyadic intervals to   give an asymptotic formula for $J_{1, d}(\fQ)$.  For applications to the results from Section~\ref{sec:Lebesgue} it is more straightforward to use a variant of Lemma~\ref{lem:aver f} with summation over intervals of the form $[1,N]$, however Lemma~\ref{lem:aver f}  is also used in the results from Section~\ref{sec:hausdroff} which require considering summation over a dyadic interval.

\begin{cor}
\label{cor:I1J1}
Let  $d \ge 2$ and let $\va=(a_n)_{n=1}^\infty$ be a sequence of complex numbers  satisfying  $|a_n|=1$. For any   interval $\fI \subseteq \TTT$ and any cube $\fQ\subseteq   \Tor$, provided $N$ is large enough in terms of $\fI$ and $\fQ$, we have  
$$
I_{1,d}(\fI) = \lambda\(\fI\) N +O\(N^{o(1)}\)  \quad \text{and} \quad  J_{1,d}(\fQ) = \lambda\(\fQ\) N +O\(N^{o(1)}\),
$$
where the implied constants
depend  only on $d$  and do not depend on $\fI$ and $\fQ$.
\end{cor}

\subsection{Fourth moments over small intervals and boxes}
\label{sec:fourth}

We now apply Lemma~\ref{lem:diag-2nu} with $\nu = 2$ to monomials of degree $d\ge 5$, to  obtain the following 
asymptotic formula for a generalisation of the  integral $I_{2,d}(\fI)$.

\begin{lemma}
\label{lem:aver 4 d>4}
Let    $\va=(a_n)_{n=1}^\infty$ be a sequence of complex numbers satisfying  $|a_n|=1$.  If $d=3$ or $d \ge 5$, then  for any interval $\fI\subseteq \TTT$ we have 
$$
\int_{\fI} \left|\sum_{n=1}^{N}a_n\e\(xn^d\)\right|^4 d x=2\lambda(\fI) N^2  +O\(N^{2-\eta_d}\), 
$$
where $\eta_d >0$ depends only on $d$ and  the implied constant may depend on $\fI$.
%% is absolute. 
\end{lemma}
 
\begin{proof} 
As in the proof of  Lemma~\ref{lem:aver f} we may suppose that $\fI=[0,\delta]$ for some $\delta \in (0,1)$.  
Using the assumption each $|a_n|=1$,  Lemma~\ref{lem:diag-2nu} implies 
\begin{equation}
\label{eq:I4 ME}
\int_{\fI} \left|\sum_{n=1}^{N}a_n\e\(x n^d\)\right|^4d x =  \sfM + \sfE, 
 \end{equation}
where 
\begin{align*}
\sfM&= \delta \sum_{\substack{  n_1, n_2, n_3, n_4\le N  \\ n_1^d +  n_2^d = n_3^d +  n_4^d}}a_{n_1}a_{n_2}  \overline {a_{n_3}} \,  \overline{a_{n_4}}, \\
\sfE& =\sum_{\substack{  n_1, n_2, n_3, n_4\le N  \\ n_1^d +  n_2^d \ne n_3^d +  n_4^d}} \frac{a_{n_1}a_{n_2}  \overline {a_{n_3}} \, \overline{a_{n_4}}\left(\e\(\delta\(n_1^d +  n_2^d - n_3^d  - n_4^d\)\)-1\right)}{2\pi i (n_1^d +  n_2^d - n_3^d  - n_4^d)}.
\end{align*}

Separating the contribution $2N^2 + O(N)$ from diagonal terms with $\{n_1, n_2\} = \{n_3,n_4\}$, thus $a_{n_1}a_{n_2}  \overline {a_{n_3}}\overline {a_{n_4}}=1$, we obtain
$$
\sfM = 2 \delta  N^2 + O\(N+T\), 
$$
where $T$ is number of solutions to the equation $n_1^d +  n_2^d = n_3^d +  n_4^d$, with  $n_1, n_2, n_3, n_4\le N $ and $\{n_1, n_2\} \ne \{n_3,n_4\}$. 
By Lemma~\ref{lem:SkinWool} (noting the comment about $d=3$), for each $d\ge 3$ there exists some $\zeta_d>0$ such that
$$
T \le N^{2-\zeta_d}
$$
which implies
\begin{equation}
\label{eq:asymp M}
\sfM = 2 \delta  N^2 + O\(N^{2-\zeta_d}\).
 \end{equation}  

To estimate $\sfE$ we write 
 $$
| \sfE| \le \sum_{\substack{-4N^d \le k \le -4N^d  \\ k\ne 0}} \frac{R_d(k,N)}{k}.
$$
where $R_d(k,N)$ is defined in Section~\ref{sec:sumdiff pow}

 In this case by Lemma~\ref{lem:Morm} for $d\ge 5$ and Lemma~\ref{lem:Hooley} for $d=3$, there exists some $\kappa_d>0$ such that 
\begin{equation}
\label{eq:bound E}
\sfE  \ll  N^{2-\kappa_d}.
\end{equation} 
Substituting~\eqref{eq:asymp M} and~\eqref{eq:bound E}  in~\eqref{eq:I4 ME} we obtain the desired bound. %% and renaming $\kappa_d$ if necessary 
\end{proof}

We now derive from Lemma~\ref{lem:aver 4 d>4} the desired bounds on $I_{2,d}(\fI)$ and $J_{2,d}(\fQ)$.

\begin{cor}
\label{cor:I2J2} 
Let  $d = 3$ or $d \ge 5$ and $\va=(a_n)_{n=1}^\infty$ a sequence of complex numbers satisfying  $|a_n|=1$. For any  
 interval $\fI \subseteq \TTT $ and any cube $\fQ\subseteq   \Tor$,  provided $N$ is large enough in terms of $\fI$ and $\fQ$,  
 we have 
$$
I_{2,d}(\fI) \ll  \lambda\(\fI\)  N^2 %+N^{2-\eta_d}  
\quad \text{and} \quad  J_{2,d}(\fQ) \ll   \lambda\(\fQ\)   N^2, %+N^{2-\eta_d} , 
$$
%where $\eta_d >0$ depends only on $d$ and 
where the implied constants are absolute. 
\end{cor}

The above leaves us with the case $d = 4$. As we have mentioned we do not have analogues of Lemmas~\ref{lem:Morm} and~\ref{lem:Hooley} 
for $d=4$. However in the case of $ J_{2,4}(\fQ) $ we are able to establish the desired result. First we obtain the following bound
for average values of exponential polynomials with quadratic amplitudes. The statement is slightly more general than we need, however we think it can be of independent interest. 

For any intervals $\fI_1, \fI_2\subseteq \TTT $ denote  
$$
\cM(\fI_1, \fI_2)=\int_{\fI_1}  \int_{\fI_2} \left|\sum_{n=1}^{N}a_n\e\(x_1n+x_2 n^2\)\right|^4 d x_1 dx_2.
$$

\begin{lemma}
\label{lem:aver quadr}
Let    $\va=(a_n)_{n=1}^\infty$ be a sequence of complex weights such that  $|a_n|=1$. For any intervals $\fI_1, \fI_2 \subseteq   \TTT $ we have 
$$
\cM(\fI_1, \fI_2)\ll \lambda(\fI_1)  \lambda(\fI_2)  N^2  + \lambda(\fI_1)^{-1}   \lambda(\fI_2)^{-1}
N^{1+o(1)},  
$$
where the implied constant is absolute. 
\end{lemma}

\begin{proof} As before, changing as needed the sequence of weights,  we can suppose that $\fI_\nu=[0,\delta_\nu]$ with some $\delta _\nu\in (0,1)$, $\nu =1,2$. 
By Lemma~\ref{lem:eqntomv}
 \begin{equation}
\label{eq:Int Qkm}
\cM(\fI_1, \fI_2)  \ll \delta_1\delta_2 \sum_{\substack{0 \le |k|\le 1/ \delta_1 \\ |m|\le 1/\delta_2}} Q(k,m,N), 
\end{equation}
where   $ Q(k,m,N)$ is the number of representations of a pair  $(k,m)$ as
$$
 k =n_1 +  n_2  - n_3 -  n_4  \mand m = n_1^2 +  n_2^2- n_3^2 -  n_4^2. 
$$
 Clearly the contribution from $(k,m) = (0,0)$ is 
\begin{equation}
\label{eq:Q00}
  Q(0,0,N) = 2N^2 + O(N).
\end{equation}

Now assume  $(k,m) \ne  (0,0)$. 

 Let $r=-n_3-k$. Eliminating $n_4$, we obtain 
$$
n_1^2+n_2^2-(r+k)^2-(n_1+n_2+r)^2=m,
$$
which is equivalent to 
\begin{equation}
\label{eq:n1n2rmk}
2(n_1+r)(n_2+r )=-m -2rk-k^2.
\end{equation}

If $m +2rk+k^2 = 0$ then $r$ is uniquely defined (using the fact $(k, m)\neq (0, 0)$), which means $n_3$ is also uniquely defined. Combining with 
$$
k=n_1+n_2-n_3-n_4,
$$
 for any $n_4$ we have at most $2$ possibilities for $(n_1, n_2)$. 
Hence in total the contribution to $Q(k,m,N)$ from such solutions is $O(N)$. 

Now we turn to the case $m+2rk+k^2\neq 0$. Note that if $|k| > 2N$ then $Q(k,m,N)=0$. 
Otherwise  $|r|\le N+|k| \le 3N$.  Since for any $r$ with $m+2rk+k^2\neq 0$, from the bound 
on the divisor function~\eqref{eq:tau}, the equation~\eqref{eq:n1n2rmk} is satisfied by at 
most  $N^{o(1)}$ pairs $(n_1,n_2)$, after which $n_4$ is uniquely defined. Therefore, the contribution from 
such solutions is $N^{1+ o(1)}$.  Hence 
\begin{equation}
\label{eq:Qkm}
  Q(k,m,N)  \le N^{1+ o(1)}, \qquad (k,m) \ne  (0,0).
\end{equation}

 Substituting~\eqref{eq:Q00} and~\eqref{eq:Qkm}  in~\eqref{eq:Int Qkm}, we obtain the desired result.
\end{proof}

Using Lemma~\ref{lem:aver quadr} and arguing as in Corollary~\ref{cor:I1J1} gives:
 \begin{cor}
\label{cor:J2}
Let  $d \ge 2$ and let $\va=(a_n)_{n=1}^\infty$ be a sequence of complex weights such that  $|a_n|=1$. For any cube $\fQ\subseteq \T_d,$ provided $N$ is large enough in terms of $\fQ$, we have 
$$
J_{2,d}(\fQ) \ll  \lambda\(\fQ\) N^2, %+  \lambda\(\fQ\)^{-2/d}  N^{1+o(1)}, 
$$
 where the implied constant is absolute.
\end{cor}

Clearly the  the bounds on $J_{2,d}(\fQ) $ from  Corollaries~\ref{cor:I2J2} and~\ref{cor:J2}  partially overlap (for $d \ge5$),
however the dependence of secondary terms on $\lambda\(\fQ\)$ is different. Both are equaly suited 
for our applications, hence for main results we only need  Corollary~\ref{cor:J2} for $d=2$ and $d=4$.

\section{Proofs of results on the Lebesgue measure}
\subsection{Proofs of Theorems~\ref{thm:mon sum} and~\ref{thm:gen sum}}

Theorem~\ref{thm:mon sum}  follows from Corollary~\ref{cor:main1}, Corollary~\ref{cor:I1J1} and Corollary~\ref{cor:I2J2}. 

In a similar fashion, Theorem~\ref{thm:gen sum} follows from Corollary~\ref{cor:main}, Corollary~\ref{cor:I1J1} and Corollary~\ref{cor:J2}.

\subsection{Proof of Theorem~\ref{thm:mon sum 4}} 
Define
$$
\cL_{N, c,C}=\left\{x\in \fI:~cN^{1/2}\le \left|\sum_{n=1}^{N} a_n \e(xn^{4})\right|\le CN^{1/2}\right\},
$$
so that 
$$
\bigcap_{k=1}^{\infty}\bigcup_{N=k}^{\infty}\cL_{N,c,C}\subseteq \cF_{\va, c}(4).
$$
By orthogonality and Lemma~\ref{lem:SkinWool}
$$
\int_{\fI} \left|\sum_{n=1}^{N} a_n \e(xn^{4})\right|^{4}dx\le \int_{\T} \left|\sum_{n=1}^{N} a_n \e(xn^{4})\right|^{4}dx\le (2+o(1))N^2.
$$
By Lemma~\ref{lem:aver f} and the above we may apply the calculations from Corollary~\ref{cor:measure} with 
$$
\alpha_1=1+o(1), \quad \alpha_2=\frac{2+o(1)}{\lambda(\fI)},
$$
to get 
$$
\lambda(\cF_{\va, c}(4))\ge \lambda(\cL_{N,c,C})\ge \left(1-c^2-\frac{2}{\lambda(\fI)C^2}+o(1)\right)\frac{\lambda(\fI)}{C^2},
$$
and the result follows taking 
$$
C^2=\frac{4}{(1-c^2)\lambda(\fI)}.
$$

\subsection{Further comments} 
\begin{remark}\label{rem:Hooley}
To extend Theorem~\ref{thm:mon sum} to include the case $d=4$, it would suffice to show that for any non-zero $k$, the number of solutions to
\begin{equation}
\label{eq:hooley4} 
x_1^4-x_2^4=x_3^4-x_4^4+k, \qquad 1 \le x_1, x_2,  x_3, x_4 \le N, 
\end{equation} 
is $o(N^2)$ as $N\to \infty$ (note that we do need any uniformity in $k$). So in particular solutions to $|x_1^4+x_2^4-x_3^4-x_4^4|\le C$ are dominated by diagonal solutions for each fixed $C>0$. It is likely that an adaption of the method of Hooley~\cite{Hool1, Hool2} on solutions to $x_1^d+x_2^d=x_3^d+x_4^d$ would yield such a result. Hooley's sieve setup~\cite{Hool1, Hool2}  generalises in a straightforward manner to handle the non-homogeneous equation~\eqref{eq:hooley4}, and reduces the question to obtaining a power-saving bound for certain complete exponential sums along a curve. Provided the exponential sum is suitably non-degenerate, variants of the Weil bound are sufficient to give such an estimate (see, for example,~\cite[Theorem~6]{Bomb}). In the interests of brevity we do not pursue this approach further here.
\end{remark}
 
 \begin{remark}
 We note that the proof of Theorem~\ref{thm:mon sum 4} actually shows that there are fixed constants $0<c<C$ such that for every choice of coefficients $\va$ with $|a_n|=1$ and every $N$, there is a set $\mathcal{S}_{\va,N}\subseteq \Tor$ of positive measure
 (bounded away from zero independently of $N$)   such that $cN^{1/2} \le |S_{\va, d}(\vx; N)|\le C N^{1/2}$ for $\vx\in\mathcal{S}_{\va,N}$. Choosing coefficients $\va$ with $|a_n|=1$ at random shows that for most choices $\va$ there are also positive measure sets for which $|S_{\va, d}(\vx; N)|< cN^{1/2}$ or $C N^{1/2} < |S_{\va, d}(\vx; N)|$, and so for individual $N$ one cannot improve the conclusion to almost all $\vx\in\Tor$.
 \end{remark}
 
\section{Some properties of exponential polynomials sums}
\label{sec:ExpPoly}

\subsection{Implied constants} 
Throughout this section, the implied constants may depend on  the function $f$, in 
particular on its smoothness and  the asymptotic behaviour  of its derivatives.
 
\subsection{Continuity of exponential polynomials} 
\label{sec:contweyl}

In full analogue of~\cite[Lemma~3.4]{ChSh-IMRN} and~\cite[Lemma~2.1]{Wool3}  we obtain:

\begin{lemma}   
\label{lem:1}
For any sequence of complex numbers $\va=(a_n)_{n=1}^\infty$ satisfying $|a_n|=1$  and any nondecreasing positive continuously differentiable  function $f(t)$, we have 
\begin{align*}
\sum_{n\le N}a_n \e\(xf(n)\)&-\sum_{n\le N}a_n \e\(yf(n)\)\\
& \ll |x-y| f(N)\max_{M\le N}\left|\sum_{n\le M}a_n \e\(xf(n)\) \right|.
\end{align*}   
%where the  implied constant depends on $f$.
\end{lemma}
\begin{proof}
Let $\delta=y-x$. We have 
$$
\sum_{n\le N}a_n \e\(xf(n)\)-\sum_{n\le N}a_n \e\(yf(n)\)
=\sum_{n\le N}\left(1-\e\(\delta f(n)\)\right)a_n \e\(xf(n)\),
$$
hence by partial summation
\begin{align*}
\sum_{n\le N}a_n &  \e\(xf(n)\)-\sum_{n\le N}a_n \e\(yf(n)\) \\ &=\left(1-\e\(\delta f(N)\)\right)\sum_{n\le N}a_n \e\(xf(n)\)\\
 & \qquad \quad +2\pi i \delta\int_{1}^{N}\e\(\delta f(t)\)f'(t)\left(\sum_{n\le t}a_n \e\(xf(n)\) \right)dt. 
\end{align*}   
Since $f$ is nondecreasing, we have 
$$1-\e\(\delta f(N)\)\ll \delta f(N) \quad  \text{and} \quad  
\int_{1}^{N}|f'(t)| dt=\int_{1}^{N}f'(t)dt \ll f(N), 
$$
 which gives the desired result. 
\end{proof}

\begin{cor} 
\label{cor:max}
For any sequence of complex numbers $\va=(a_n)_{n=1}^\infty$ satisfying $|a_n|=1$, any nondecreasing positive continuously differentiable  function $f(t)$ and any real numbers $x,y$ satisfying $|x-y|\ll f(N)^{-1}$, we have  
$$
\max_{M\le N}\left|\sum_{n\le M}a_n \e\(xf(n)\) \right|  \asymp \max_{M\le N} \left|\sum_{n\le M}a_n \e\(yf(n)\) \right|.
$$
%where the  implied constants depend on $f$. 
\end{cor}
\begin{proof}
For any $M\le N$ applying Lemma~\ref{lem:1} we have
\begin{align*}
\left|\sum_{n\le M}a_n \e\(xf(n)\) \right| &\ll |x-y| f(M)  \max_{K\le M}\left|\sum_{n\le K}a_n \e\(yf(n)\) \right|\\
& \ll \max_{K\le N}\left|\sum_{n\le K}a_n \e\(yf(n)\) \right|.
\end{align*} 
By the arbitrary choice of $M\le N$ we obtain 
$$
\max_{M\le N}\left|\sum_{n\le M}a_n \e\(xf(n)\) \right|  \ll \max_{M\le N} \left|\sum_{n\le M}a_n \e\(yf(n)\) \right|.
$$
The other   inequality follows from symmetry.
\end{proof}

\subsection{Variance of mean values} 
%\label{subsec:variance}

Our main technical tool in proving  Theorem~\ref{thm:f}  is the following asymptotic formula for moments of exponential sums.  We remark that  we do not need this 
for the proof of Theorem~\ref{thm:d=2}. 

\begin{lemma}  
\label{lem:variance}
Let $f$ be a real function with continuous second derivative and satsifying
$$
 f''(x)= x^{\gamma-2+o(1)},   
$$
for some  $\gamma >2$.  
Let  $\varepsilon_0$, $\varepsilon_1$, $x_1$ be real numbers. 
For any sequence  $\va=(a_n)_{n=1}^\infty$ of complex numbers  satisfying  $|a_n|=1$,  $N\in \N$ and $ M=\fl{N/2}$, 
for 
$$
\sI(M,N) =  \int_{x_1}^{x_1+\varepsilon_1}  \(\int_{x_0}^{x_0+\varepsilon_0}\left|\sum_{M<n\le N}a_n \e\(xf(n)\)\right|^2dx-\varepsilon_0 N \)^2dx_0 
$$
we have 
$$
\sI(M,N)  \le N^{-2\gamma+3+o(1)}\left(\varepsilon_1+N^{-\gamma+2}\right).
$$ 
\end{lemma}

\begin{proof}
Applying Lemma~\ref{lem:diag-2nu}   with $\nu=2$ and separating the contribution from diagonal terms gives
\begin{align*}
& \int_{x_0}^{x_0+\varepsilon_0}\left|\sum_{M<n\le N} a_n \e\(xf(n)\)\right|^2dx-\varepsilon_0(N-M)  \\ 
& \qquad \qquad   \ll\sum_{\substack{ M\le n<m \le N \\ m\neq n }}\frac{a_{m}\overline a_{n}
\left(\e\(\varepsilon_0(f(m)-f(n))\)-1\right)}{ f(m)-f(n)}\\
& \qquad \qquad \qquad \qquad \qquad \qquad \qquad \qquad \qquad \quad \times \e\(x_0(f(m)-f(n))\) \\
 & \qquad \qquad  \ll \sum_{1\le h \le N}\frac{1}{h}\left|\sum_{\substack{M<n\le N-h }} \frac{\beta_{n,h}}{\Delta_h(n)}\e\(x_0h\Delta_h(n)\) \right|, 
\end{align*} 
where we have made the change of variable $m \rightarrow n+h$ and defined
\begin{align*}
& \Delta_h(n)=(f(n+h)-f(n))/h, \\
& \beta_{n,h}=a_{n+h}\overline a_{n}(\e\(\varepsilon_0(h\Delta_h(n))\)-1).
\end{align*} 
Squaring, applying the Cauchy-Schwarz inequality then integrating over $(x_1,x_1+\varepsilon_1)$ gives 
\begin{equation}
\label{eq:st1}
\sI(M,N)  \ll \log{N}\sum_{1\le h \le N}\frac{I_h}{h},
\end{equation}
where 
$$
I_h = \int_{x_1}^{x_1+\varepsilon_1}\left|\sum_{M< n \le N-h}\frac{\beta_{n,h}}{\Delta_h(n)}
\e\(x_0h\Delta_h(n)\) \right|^2dx_0. 
$$
A second application of Lemma~\ref{lem:diag-2nu} (again with $\nu=1$) and using that $|\beta_{n,h}|\ll 1$, yields
\begin{equation}
\begin{split}
\label{eq:one-h}
I_h &  \ll   \varepsilon_1\sum_{M<n\le N}\frac{1}{\Delta_h(n)^2} \\
 &\qquad \qquad  + \sum_{M<m< n\le N }\frac{1}{\Delta_h(m)\Delta_h(n)}\frac{1}{|\Delta_h(n)-\Delta_h(m)|}.
\end{split}
\end{equation} 
By the mean value theorem, we have  
$$
\Delta_h(n)= f'(\xi), \quad \text{for some $\xi\in [n,n+h]$}.
$$ 
The assumptions  on $f''(t)$ imply that  $f'(t)= t^{\gamma-1+o(1)}$ for $t$ sufficiently large. Since we can clearly assume that $N$ is large enough in terms of $f$, we obtain \begin{equation}
\label{eq:>>1}
\Delta_h(n)\ge n^{\gamma-1+o(1)}.
\end{equation}
Now applying  the mean value theorem twice we obtain 
$$
\Delta_h(n)-\Delta_h(m)=(n-m)\Delta'_h(z), \quad \text{for some $z\in [m,n]$},
$$  
and  
$$
\Delta_h'(z)=(f'(z+h)-f'(z))/h=f''(z_0), \quad \text{for some $z_0\in [z,z+h].$}
$$
Then recalling $f''(t)= t^{\gamma-2+o(1)} $ we get 
\begin{equation}
\label{eq:>>2}
\Delta_h(n)-\Delta_h(m)\ge (n-m)m^{\gamma-2+o(1)}.
\end{equation}
Now, using~\eqref{eq:>>1} and~\eqref{eq:>>2},  we derive 
$$
\sum_{M<n\le N}\frac{1}{\Delta_h(n)^2} \le  \sum_{M<n\le N} n^{-2\gamma+2+o(1)} = N^{-2\gamma+3+o(1)}
$$
and 
\begin{align*}
 \sum_{M< m< n\le N } & \frac{1}{\Delta_h(m)\Delta_h(n)}\frac{1}{|\Delta_h(n)-\Delta_h(m)|}\\ 
&  \le  N^{o(1)}  \sum_{M< m< n\le N }  \frac{1}{m^{\gamma-1} n^{\gamma-1}} \cdot \frac{1}{ (n-m)m^{\gamma-2}}\\ 
& \qquad \le  N^{o(1)}  \sum_{M< m< n\le N }  \frac{1}{m^{3\gamma-4} } \cdot \frac{1}{n-m}\\ 
& \qquad \le  N^{o(1)}  \sum_{m=M+1 }^N  \frac{1}{m^{3\gamma-4} }   \sum_{n=m+1 }^N   \frac{1}{n-m}\le  N^{-3\gamma+5 + o(1)} . 
\end{align*} 
Substituting these inequalities in~\eqref{eq:one-h} gives
$$
I_h   \le  N^{-2\gamma+3+o(1)}\left(\varepsilon_1+N^{-\gamma+2}\right)
$$
and combined   with~\eqref{eq:st1}  yields the desired bound.
\end{proof}

The next result is our main tool in proving Theorem~\ref{thm:f}. 
 For two intervals $\cI$ and  $\cJ$ let $\text{Dist}(\cI, \cJ)$ denote the gap between them, that is, 
$$
\text{Dist}(\cI, \cJ)=\inf \{\|x-y\|:~x\in \cI,\ y\in \cJ\}.
$$ 

We say that  two intervals $\cI$ and  $\cJ$ are {\it $\Delta$-separated\/} if 
$$
\text{Dist}(\cI, \cJ)\ge \Delta.
$$

\begin{lemma}
\label{lem:largeshort}
Let $f$ satisfy the conditions of Lemma~\ref{lem:variance}. Let $\tau>0$ be a small parameter and let  $\va=(a_n)_{n=1}^\infty$  be a sequence of complex weights satisfying $|a_n|=1$. 
For any interval $\cI\subseteq \TTT $ and  for all large enough $N$ with
$$
|\cI|\ge  N^{-\gamma+2}
$$ 
 there exists a collection of 
$$
 K \gg N^{\gamma-1/2-\tau}|\cI|
 $$ 
 pairwise $N^{-\gamma+1/2+\tau}$-separated intervals $\cI_{i} \subseteq \cI$,  $1\le i\le K$, such that 
$$
  |\cI_{i}|=  N^{-\gamma+1/2+\tau}
$$
and 
\begin{equation}
\label{eq:xconds}
\max_{x\in \cI_i}\left|\sum_{\fl{N/2}\le n\le N}a_n \e\(xf(n)\)\right|\gg N^{1/2}.  
\end{equation}
\end{lemma}

\begin{proof} 
Let
$$\cI=[x_1,x_1+\varepsilon_1],$$
for some $x_1,\varepsilon_1$ with 
$$\varepsilon_1 = N^{-\gamma+2+\tau}.$$
Applying Lemma~\ref{lem:variance} with   
\begin{equation}
\label{eq:vareps0}
\varepsilon_0=  N^{-\gamma+1/2+\tau},
\end{equation}
we obtain 
\begin{equation}
\begin{split}
\label{eq:ub1}
 \int_{\cI}\(\int_{x_0}^{x_0+\varepsilon_0}\left|\sum_{M< n \le N}a_n \e\(xf(n)\)\right|^2dx-\varepsilon_0(N-M) \)^2dx_0& \\
 \le N^{-2\gamma+3+o(1)}&|\cI|.
\end{split}
\end{equation}
Suppose $\varepsilon>0$ is small and let $\cS\subseteq \cI$ denote the set of $x_0$  satisfying 
$$
\left|\int_{x_0}^{x_0+\varepsilon_0}\left|\sum_{M< n \le N}a_n \e\(xf(n)\)\right|^2dx-\varepsilon_0(N-M) \right| 
\ge  N^{-\gamma+3/2+\varepsilon}.
$$ 
The Cauchy-Schwarz inequality and~\eqref{eq:ub1} imply  
\begin{align*}
& \left(  \lambda(\cS) N^{-\gamma+3/2+\varepsilon}\right)^2\\
&\quad \  \le \lambda(\cS)\int_{\cI}\(\int_{x_0}^{x_0+\varepsilon_0}\left|\sum_{M< n \le N}a_n \e\(xf(n)\)\right|^2dx-\varepsilon_0(N-M)\)^2dx_0 \\
&\quad  \le N^{-2\gamma+3+o(1)}|\cI|  \lambda(\cS).
\end{align*}    
For sufficiently large $N$ this gives 
$$
 \lambda(\cS) \le \frac{N^{o(1)}|\cI|}{N^{2\varepsilon}}\le \frac{|\cI|}{2}.
$$
Hence for the set  $\cA=\{ x\in \cI:~x\not \in \cS\}$ we have 
\begin{equation}
\label{eq:Bub}
 \lambda(\cA) \ge \frac{|\cI|}{2}.
\end{equation}
With $\varepsilon_0$ as in~\eqref{eq:vareps0}, for each $\alpha \in \cA$ let $\cB_\alpha$ denote the interval 
$$\cB_\alpha=[\alpha,\alpha+\varepsilon_0]$$
so that 
$$
\cA\subseteq \bigcup_{\alpha \in \cA}\cB_\alpha.
$$

For an interval $\cJ=[x-r, x+r]$ denote $\cJ^{\times 5}=[x-5r, x+5r]$ its $5$-fold blow-up.  
Applying  the Vitali Covering Theorem~\cite[Theorem~1.24]{EG}   to the collection $\cB_\alpha$, $\alpha\in \cA$, there exists a subset $\cA_1\subseteq \cA$ such that 
\begin{equation}
\label{eq:5r}
\cA\subseteq \bigcup_{\alpha \in \cA}\cB_\alpha \subseteq \bigcup_{\alpha \in \cA_1}   \cB_\alpha^{\times 5} 
\end{equation}
and for all $\alpha, \beta \in \cA_1$ with $\alpha\neq \beta$ we have 
$\cB_{\alpha} \cap \cB_{\beta} \neq \emptyset$. Combining~\eqref{eq:Bub} with~\eqref{eq:5r} we conclude  
\begin{equation}
\label{eq:5r-measure}
|\cI|\ll \left|\bigcup_{\alpha \in \cA_1}   \cB_\alpha^{\times 5} \right|\ll \sum_{\alpha \in \cA_1}|\cB_\alpha|.
\end{equation}
It follows that $\cA_1$ is a finite set. 
Note that  there exists a subset $\cA_2\subseteq \cA_1$ such that  $\#\cA_2\gg \#\cA_1$  and for all $\alpha, \beta\in \cA_2$ with $\alpha \neq \beta$ we have 
$$
\text{Dist} (\cB_\alpha, \cB_\beta)\ge N^{-\gamma+1/2+\tau},
$$
 which establishes the desired $N^{-\gamma+1/2+\tau}$-separation.    
Combining this with~\eqref{eq:5r-measure} we derive
$$
|\cI|\ll \sum_{\alpha \in \cA_2}|\cB_\alpha| \ll N^{-\gamma+1/2+\tau}\# \cA_2,
$$
 which establishes the desired bound on 
 $$
 K = \# \cA_2 \gg N^{\gamma-1/2-\tau}|\cI|. 
 $$

It remains to show~\eqref{eq:xconds}. Let $\alpha\in \cA_2$ then 
$$
\left|\int_{\alpha}^{\alpha+\varepsilon_0}\left|\sum_{M< n \le N}a_n \e\(xf(n)\)\right|^2dx-\varepsilon_0(N-M) \right|\le   N^{-\gamma+3/2+\varepsilon}.
$$
Recalling the choice of $\varepsilon_0$  in~\eqref{eq:vareps0} and that $M=\fl{N/2}$,  
after choosing $\varepsilon<\tau$, for large enough $N$ we obtain 
$$
\varepsilon_0(N-M)\ge 2N^{-\gamma+3/2+\varepsilon}
$$
and hence we conclude 
\begin{align*}
\varepsilon_0\max_{x\in I_{\alpha}} & \left|\sum_{M< n \le N}a_n \e\(x f(n)\)\right|^2\\
&  \ge \int_{\alpha}^{\alpha+\varepsilon_0}\left|\sum_{M< n \le N}a_n \e\(xf(n)\)\right|^2dx\gg \varepsilon_0 N.
\end{align*}
Changing the numbering of intervals $\cB_\alpha$ from elements of $\cA_2$ to $\cB_i$, 
$i=1, \ldots, K$, $K =\#\cA_2$ we  complete the proof. 
\end{proof}

\section{Proofs of results on Hausdorff dimension} 

\subsection{Hausdorff dimension of a class of Cantor sets}
\label{sec:Cantor}

A typical way to obtain a lower bound for the Hausdorff dimension of some given set is to determine the Hausdorff dimension of a Cantor-like subset via the mass distribution principle, see~\cite[Chapter~4]{Falconer}.

Here we introduce  a class of Cantor sets which is motivated by iterating the construction of  Corollary~\ref{cor:max}. For convenience we introduce the following definition.

\begin{definition}[$\cI(N, M, \delta)$-patterns] \label{def:p}
Given an interval $\mathcal{I}$, integers $1\le M\le N$ with $N\ge 2$ and a constant $\delta>0$, an $\mathcal{I}(N,M,\delta)$-pattern is a set $\mathcal{S}=\{\mathcal{I}_k:\, 1\le k\le M\}$ of $M$ intervals such that
\begin{enumerate}
\item Each interval $\mathcal{I}_k\in\mathcal{S}$ has length $\delta$.
\item If $\mathcal{I}$ is split into $N$ distinct subintervals of equal length, then each $\mathcal{I}_k\in\mathcal{S}$ is contained in one of these subintervals, and no subinterval contains more than one element of $\mathcal{S}$.
\end{enumerate} 
\end{definition}

Figure~\ref{fig:patter} gives a visual example of an $\mathcal{I}(N,M,\delta)$-pattern.

\begin{figure}[H]
\includegraphics[scale=0.6]{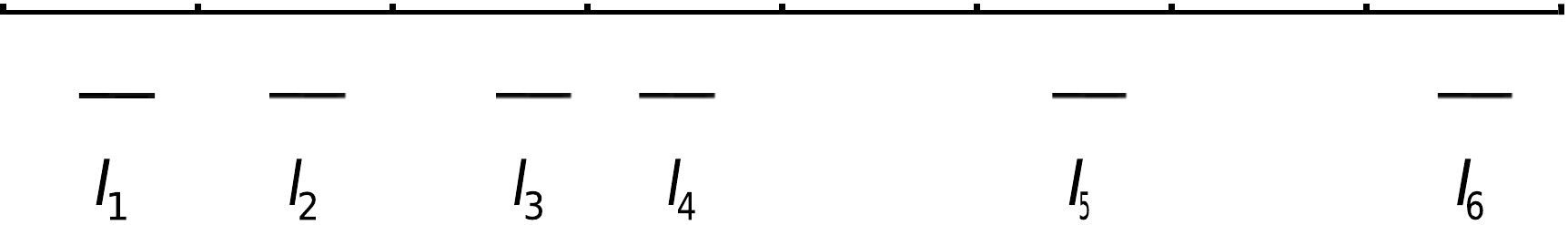}
\caption{A sample of  the $\cI(N, M, \delta)$-pattern with $N=8$, $M=6$ and some positive $\delta$. The collection of the intervals $\cI_i$, $1\le i\le 6$, forms the $\cI(8, 6, \delta)$-pattern.} 
 \label{fig:patter}
\end{figure} 

We remark that for our setting the exponential sums have large values at the intervals $\cI_i$, $1\le i\le 4$,  of Figure~\ref{fig:patter}.
Moreover for each interval $\cI_i$, $1\le i\le 4$, there are some subintervals which admits large exponential sums as well.  Thus by the iterated construction the exponential sums have large values on a Cantor-like set, and therefore this gives the lower bounds  of Theorem~\ref{thm:f} and Theorem~\ref{thm:d=2}.

\begin{remark}
\label{rem:J-patt}
 We also use the notation $\cJ(N, M, \delta)$ when the above process is applied to the interval $\cJ$. 
\end{remark}

We  construct  Cantor sets by iterating the above $\cI(N, M, \delta)$-patterns.  

Let $(M_k)$ and $(N_k)$ be two sequence natural numbers with $1\le M_k\le N_k$  and $N_k\ge 2$ for all $k\in \N$. Let $(\delta_k)$ be a sequence of positive numbers with $\delta_0=1$ and $\delta_{k}\le \delta_{k-1}/N_{k}$ for all $k\in \N$.

We start from an interval $\cI_0$ and take a $\cI_0(N_1, M_1, \delta_1)$-pattern inside of $\cI_0$. Let $\fC_1$ be the collection of these $M_1$-subintervals. More precisely,  let 
$$
\fC_1=\{\cI_i:~1\le i\le M_1\}.
$$
Note that each subinterval $\cI_i$, $1\le i\le M_1$, has length $\delta_1$. For each $\cI_i$ we take an $\cI_i(N_2, M_2, \delta_2)$-pattern inside of $\cI_i$, and we denote these subintervals of $\cI_i$ by $\cI_{i, j}$ with $1\le j\le M_2$. Let 
$$
\fC_2=\{\cI_{i, j}:~1\le i\le M_1,\ 1\le j\le M_2\}.
$$

Note that  the choices of $\cI_i(N_2, M_2, \delta_2)$-pattern and $\cI_j(N_2, M_2, \delta_2)$-pattern are independent for $i\neq j$.

Suppose that we have $\fC_k$ which is a collection of 
$$
\#\fC_k =  \prod_{i=1}^{k} M_k
$$ 
intervals of length $\delta_k$. For each of these intervals $\cI \in \fC_k$ we select a $\cI(N_{k+1}, M_{k+1}, \delta_{k+1})$-pattern inside of  $\cI$. Let $\fC_{k+1}$ be the collection of these intervals, that is 
$$
\fC_{k+1}=\{\cI_{i_1, \ldots, i_{k+1}}:~1\le i_1\le M_1, \ldots, 1\le i_{k+1}\le M_{k+1}\}.
$$ 
Our Cantor-like set is defined by 
$$
\cF=\bigcap_{k=1}^{\infty} \cF_k,
$$ 
where 
$$
\cF_k=\bigcup_{\cI\in \fC_k} \cI.
$$

There are uncountably many possible configurations for the above construction, we let $\Omega((N_k), (M_k), (\delta_k))$ denote
the collections of all the possible configurations.

For determining the Hausdorff dimension of such a set, we  use  the following  {\it mass distribution principle\/}, 
see~\cite[Theorem~4.2]{Falconer}.

\begin{lemma}
\label{lem:ms}
Let $\cX\subseteq \R$ and let $\mu$ be a Borel measure on $\R$ such that 
$$
\mu(\cX)>0.
$$
If there exist  $c$ and $\delta$ such  that for any interval $B(r)$ of length $r$ with $0<r<\delta$ we have 
$$
\mu(B(r))\le c r^{s},
$$
then $\dim \cX\ge s$.
\end{lemma}

We believe the following general result is of independent interest and may find some 
other applications.

\begin{lemma}
\label{lem:dimsion-Cantor}
Using the above notation, suppose that 
$$
M_k\ge  c N_k, \quad k\in \N, 
$$
for some absolute  constant $c>0$. Then for any set
$$ 
\cF\in \Omega((N_k), (M_k), (\delta_k))
$$ 
we have 
$$
\dim \cF= \liminf_{k\rightarrow \infty}\frac{\log \prod_{i=1}^{k} M_i}{\log (1/\delta_k)}.
$$
\end{lemma}

\begin{proof}
It is convenient to define
$$
P_k = \prod_{i=1}^{k} M_i.
$$
Let 
$$
s=\liminf_{k\rightarrow \infty}\frac{\log P_k}{\log (1/\delta_k)}.
$$
For any $\varepsilon>0$ there exists a subsequence $k_n$,  $n\in \N$ such that 
\begin{equation}
\label{eq:upper}
P_{k_n}   \le \delta_{k_n}^{-s-\varepsilon}
\end{equation}
for all large enough $n$.

Observe that for each $k_n$ the set $\cF$ is covered by   $P_{k_n} $ intervals and each of them has length $\delta_{k_n}$. Combining with~\eqref{eq:upper} we have 
$$
\delta_{k_n}^{s+2\varepsilon} P_{k_n}  \le \delta_{k_n}^{\varepsilon}.
$$
Thus the definition of Hausdorff dimension implies that $\dim \cF\le s+2\varepsilon$. By the arbitrary choice of $\varepsilon>0$ we obtain that $\dim \cF \le s$.

Now we  use the {\it mass distribution principle\/}  to obtain a lower bound for $\dim E$.  Thus we first construct a measure on $\cF$. For each $k$ let $\nu_k$ be a probability measure on $\TTT $ such that 
$$
\nu_k(\cI)=\frac{1}{\# \fC_k}=P_k^{-1},\qquad \forall \cI\in \fC_k,
$$
where $\fC_k$ is  the corresponding  collection of $\# \fC_k = P_k$ intervals as in the above. 
The measure $\nu_k$ weakly converges to a measure $\mu$, see~\cite[Chapter~1]{Mattila1995}.

Let $0<t<s$ then for all large enough $k$ we have 
\begin{equation}
\label{eq:k-large}
P_{k}  \ge \delta_k^{-t}.
\end{equation}

For any interval  $B(r)$ with $0<r<1$ there exists $k\in \N$ such that 
$$
\delta_{k+1}< r\le \delta_k. 
$$

Since the value $\delta_{k+1}$ maybe quite smaller than  the value $\delta_{k}$, we do a case by case argument according to the value of $r$.

{\bf Case 1:} Suppose that $\delta_k/N_{k+1}\le  r<\delta_k$. Since the interval $B(r)$ intersects at most  $3 r N_{k+1}/\delta_k$  disjoint intervals of equal length $\delta_k/N_{k+1}$, and inside each of these intervals there exists at most one  interval of $\fC_{k+1}$, we obtain that 
$$
\nu_{k+1} (B(r))\ll \frac{ r N_{k+1}}{\delta_k P_{k+1} }.
$$

Applying the condition $M_k\ge c N_k$, the estimate~\eqref{eq:k-large} and the assumption $r<\delta_k$,  we obtain 
$$
\nu_{k+1} (B(r))\ll \frac{r}{\delta_k P_k}  \ll \frac{r}{\delta_k} \delta_k^{t}= r \delta_k^{t-1} \ll r^{t}.
$$

{\bf Case 2:} Suppose that $\delta_{k+1}\le  r\le \delta_k/N_{k+1}$. Note that the interval $B(r)$ intersects at most two intervals with equal length $\delta_k/N_{k+1}$ and thus meets at most two intervals of $\fC_{k+1}$. Combining with~\eqref{eq:k-large} and the assumption $\delta_{k+1}\le  r$, we have 
$$
\nu_{k+1}(B(r))\le \frac{2}{P_{k+1}}\ll \delta_{k+1}^{t}\le r^{t}.
$$ 

Putting {\bf Case 1} and {\bf Case 2} together, we conclude that
\begin{equation}
\label{eq:nu-ball}
\nu_{k+1}(B(r))\ll r^{t}.
\end{equation}

Note that for $\delta_{k+1}\le  r<\delta_k$ we have 
$$
\mu(B(r))\le \nu_{k+1} (B(3r)).
$$ 
By~\eqref{eq:nu-ball} we obtain 
$\mu(B(r))\ll r^{t}$. 
Applying Lemma~\ref{lem:ms}, we arrive at $\dim \cF\ge t$. By the arbitrary choice of $t<s$ we obtain that $\dim \cF \ge s$, which finishes the proof.
\end{proof}

We remark that the condition $M_k\ge cN_k, k\in \N$ appears naturally in the proofs of Theorems~\ref{thm:f} and~\ref{thm:d=2}.  Moreover, the dimension formula of Lemma~\ref{lem:dimsion-Cantor} may not hold in general 
without the condition $M_k\ge cN_k$, $k\in \N$. However, there are upper bounds and lower bounds for the general situation and more general constructions of Cantor-like sets, see~\cite{FWW} for more details.

We now formulate the following result which fits into  our application immediately. 

\begin{cor}
\label{cor:diemension}
Using above notation, suppose that 
$$
M_k\ge  c N_k, \quad k\in \N
$$
for some constant $c>0$, and $M_k$ tends to infinity  rapidly such that 
$$
\lim_{k\rightarrow \infty}\frac{\log \prod_{i=1}^{k-1}M_i }{\log M_k}=0.
$$
Then for any $\cF\in \Omega((N_k), (M_k), (\delta_k))$ we have 
$$
\dim \cF= \liminf_{k\rightarrow \infty}\frac{\log M_k}{\log (1/\delta_k)}.
$$
\end{cor}

\subsection{Proof of Theorem~\ref{thm:f}} 

Let $f$ satisfy the conditions of Theorem~\ref{thm:f} and let $\cF_{\va, c}(f)$ denote the set of $x\in \cI$ such that 
$$
 \left|\sum_{1\le n \le N} a_ne(x f(n))  \right|\ge c N^{1/2}  \quad \text{for infinitely many $N\in \N$}.
$$

We  construct  a Cantor set inside  $\cF_{\va, c}(f)$  then apply results of Section~\ref{sec:Cantor} to obtain the 
desired lower bound of $\dim \cF_{\va, c}(f)$.

For the construction of the Cantor set, we start from an arbitrary interval $\cI\subseteq \R$ and some large number $N$. Applying Lemma~\ref{lem:largeshort} to the interval $\cI$ and the number $N$, we  obtain a collection  (taking $M_1$ instead of $K$)  
of 
\begin{equation}
\label{eq:M1c}
 M_1 \gg N^{\gamma-1/2-\tau}|\cI|
 \end{equation}
 pairwise $N^{-\gamma+1/2+\tau}$-separated intervals   $\cI_i$, $1\le i\le M_1$,    satisfying
$$
|\cI_i|=N^{-\gamma+1/2+\tau}
$$
such that there exists some $x_i\in \cI_i$ with 
\begin{equation}
\label{eq:QN1}
\left |\sum_{\fl{N/2}\le n\le N} a_n\e\(x_i f(n)\) \right | \gg N^{1/2}.
\end{equation} 
Note that for any complex numbers $a$ and $b$, by the triangle inequality, we have 
$$
\max\{|a|,|b|\} \ge \max\{ |a-b| - |b|, |b|\} \ge  |a-b|/2.
$$
Hence, the inequality~\eqref{eq:QN1} implies
\begin{equation}
\begin{split}
\label{eq:QN}
\max_{Q\le N} &\left |\sum_{n \le Q} a_n\e\(x_i f(n)\) \right | \\
&\qquad \ge \max\left\{\left |\sum_{n \le N} a_n\e\(x_i f(n)\) \right |,\left |\sum_{n \le N/2} a_n\e\(x_i f(n)\) \right |\right\}\\
&\qquad \ge \frac{1}{2} \left |\sum_{N/2< n\le N} a_n\e\(x_i f(n)\) \right | 
%%+\left |\sum_{n=1}^{N/2} a_n\e\(x_i f(n)\) \right |\\
 \gg N^{1/2}.
\end{split} 
\end{equation}

Furthermore, since the intervals $\cI_i$, $1\le i\le M_1$, are $N^{-\gamma+1/2+\tau}$-separated, that is 
$$
\text{Dist} (\cI_{i}, \cI_{j})\ge N^{-\gamma+1/2+\tau}, \qquad  1\le i <j \le M_1, 
$$
we obtain that 
\begin{equation}
\label{eq:point-s}
|x_i-x_j|\ge  N^{-\gamma+1/2+\tau}, \qquad  1\le i <j \le M_1.
\end{equation}

We now set 
\begin{equation}
\label{eq:N1def}
N_1= \rf{N^{\gamma-1/2-\tau}} + 1
\end{equation}
and divide the interval $\cI$ into $N_1$ subintervals of equal length $N_1^{-1}$. 
Note that the choice of  $N_1$ makes sure that the length of the subinterval is slightly smaller than 
 $N^{-\gamma+1/2+\tau}$.  
 
 For each $1\le i\le M_1$, among the above $N_1$ subintervals there is an interval $\cJ_i$  containing $x_i$. Indeed if $x_i$ meets two of them then we choose one only. By~\eqref{eq:point-s} we conclude that $\cJ_k$ and $\cJ_\ell$ are separated for all $1\le k<\ell\le M_1$. In fact what we need in the following construction is that $\cJ_k\neq \cJ_\ell$ for $1\le k<\ell\le M_1$.

For each $\cJ_i$, the estimate~\eqref{eq:QN} 
and Corollary~\ref{cor:max} imply that there exists a subinterval 
$\widetilde{\cJ}_i\subseteq \cJ_i$ with length $\delta_1=N^{-\gamma-\tau}$ such that 
$$
\max_{Q\le N}\left |\sum_{n=1}^{Q} a_n\e\(x f(n)\) \right | \gg N^{1/2}, \qquad \forall x\in \widetilde{\cJ}_i.
$$

Note that the collection of intervals $\widetilde{\cJ}_i$, $1\le i\le M_1$, forms a $\cI(N_1, M_1, \delta_1)$-pattern 
as in Definition~\ref{def:p}.

Let 
$$
\fC_1=\{\widetilde{\cJ}_i:~i=1, \ldots, M_1\}.
$$

Moreover, by~\eqref{eq:M1c} and~\eqref{eq:N1def} we have  $M_1 \gg N_1$ where the implied constant depends on $\cI$.

Let $\cF_1$ be the union of intervals of $\fC_1$. The set  $\cF_1$ is the first step  in the  construction of the desired Cantor-like set, see Figure~\ref{fig:f} for  the case $M_1=3$.  

Suppose we have constructed a sequence $\fC_1, \ldots, \fC_k$ where $\fC_k$ is a union of disjoint intervals $\cI_i$, $1\le i\le \# \fC_k$, of equal length $\delta_k$. We  next construct a set $\fC_{k+1}$ which is a union of disjoint intervals of equal length $\delta_{k+1}$ for suitable $\delta_{k+1}$.

Let  $L_k$ satisfy 
\begin{equation}
\label{eq:con1}
\delta_k\ge L_k^{-\gamma+2},
\end{equation}
which is chosen so our parameters in the construction of $\fC_{k+1}$ satisfy the conditions of Lemma~\ref{lem:largeshort}. 
For each interval $\cJ\in \fC_k$,  we use a similar argument to the above construction of $\fC_1$. To be  precise, let 
$$
N_{k+1}=\ceil{\delta_k L_k^{\gamma-1/2-\tau}}+1.
$$
We divide the interval $\cJ$ into $N_{k+1}$ subintervals of equal length $\delta_k N_{k+1}^{-1}$. Note that  the choice of $N_{k+1}$ make sure that the length of the subinterval is slightly smaller than $L_k^{-\gamma+1/2+\tau}$.  

For the interval $\cJ$ and  $L_k$,
applying Lemma~\ref{lem:largeshort}, we conclude that  among these  
$N_{k+1}$ intervals, there are $M_{k+1}$ intervals  $\cJ_{\cI, 1}, \ldots, \cJ_{\cI, M_{k+1}}$ of length $L_k^{-\gamma+1/2+\tau}$ such that for each $1\le \ell \le M_{k+1}$ there is a $x_\ell\in \cJ_{\cI, \ell}$ satisfying 
$$
\max_{Q\le L_k}\left |\sum_{n=1}^{Q} a_n\e\(x_\ell f(n)\) \right | \gg L_k^{1/2}.
$$

Furthermore, 
$$
N_{k+1}\ge M_{k+1}\gg L_k^{\gamma-1/2-\tau} \delta_k \gg N_{k+1}.
$$

For each $x_\ell$, $1\le \ell \le M_{k+1}$, by Corollary~\ref{cor:max} there exists a subinterval $\widetilde{\cJ_{\cI, \ell}}\subseteq \cJ_{\cI, \ell}$ such that 
$$
|\widetilde{\cJ_{\cI, \ell}}|=\delta_{k+1}=L_k^{-\gamma-\tau}
$$ 
and 
\begin{equation}
\label{eq:Q}
\max_{Q\le L_k}\left |\sum_{n=1}^{Q} a_n\e\(x f(n)\) \right | \gg L_k^{1/2}, \qquad \forall x\in \widetilde{\cJ_{\cI, \ell}}.
\end{equation} 
Thus the collection of intervals $\widetilde{\cJ_{\cI, \ell}}$ forms a $\cJ(N_{k+1}, M_{k+1}, \delta_{k+1})$ pattern.  Note that for $\cJ_1, \cJ_2\in \fC_k$ with $\cJ_1 \neq \cJ_2$ the two patterns $\cJ_1(N_{k+1}, M_{k+1}, \delta_{k+1})$ and $\cJ_2(N_{k+1}, M_{k+1}, \delta_{k+1})$ may be different in general.

Let $\fC_{k+1}$ be the collection of these $\cJ(N_{k+1}, M_{k+1}, \delta_{k+1})$ patterns with $\cJ\in \fC_k$.  Our desired Cantor set is defined as 
$$
\cF=\bigcap_{k=1}^{\infty} \cF_k,
$$
where
$$
\cF_k=\bigcup_{\cI\in \fC_k} \cI.
$$

\begin{figure}[H]
\centering 
\includegraphics[scale=0.6]{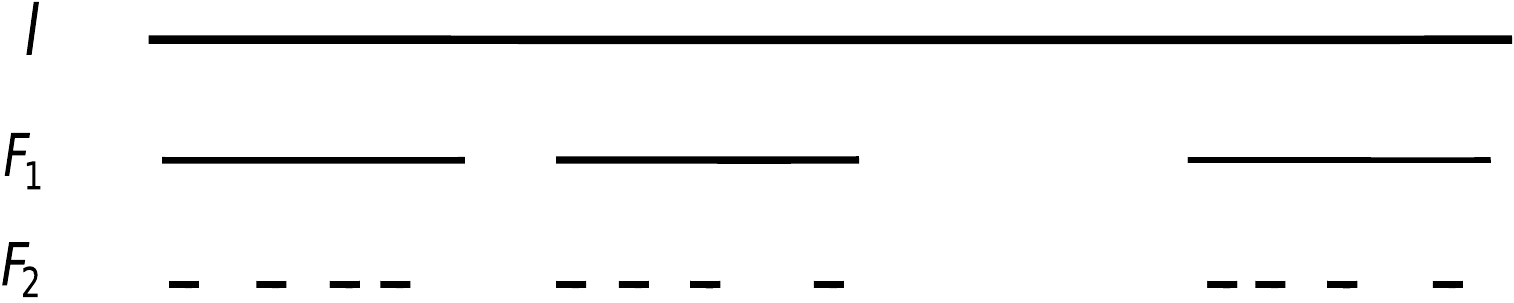}
\caption{Two steps construction of the Cantor-like set with $M_1=3$ and $M_2=4$.}
 \label{fig:f}
\end{figure}

Note that the set $\cF$ is an element of $\Omega((N_k), (M_k), (\delta_k))$ as defined in Section~\ref{sec:Cantor}.
Now we are going to show that  
\begin{equation}
\label{eq:subset}
\cF\subseteq \cF_{\va, c}(f)
\end{equation}
for some choices of parameters   $N_k$, $M_k$ and $\delta_k$, where $k\in \N$.   

Let $x\in \cF$ then $x\in  \cF_{k}$ for all $k\in \N$.   
The estimate~\eqref{eq:Q} implies that  there exists $Q_k$ such that 
$$
L_k^{1/2}\ll Q_k\le L_k,
$$
and 
$$
\left |\sum_{n=1}^{Q_k} a_n\e\(x f(n)\) \right| \gg | Q_k|^{1/2}.
$$
For each $k$ we choose $L_k$ large enough such that 
\begin{equation}
\label{eq:con2}
Q_1<Q_2<\ldots,
\end{equation}
which implies 
$$
\sum_{n=1}^{Q}a_n \e\(xf(n)\) \gg Q^{1/2} 
$$
for infinitely many  $Q\in \N$ 
and hence we have~\eqref{eq:subset}.  For each $k$ we can choose $L_k$ even larger such that the conditions~\eqref{eq:con1},~\eqref{eq:con2} hold, and 
$$
\lim_{n\rightarrow \infty}\frac{\log \prod_{i=1}^n N_i}{ \log N_{n+1}}=0,
$$
Clearly the condition $N_k \asymp M_k$, $k\in \N$ implies  
$$
\lim_{n\rightarrow \infty}\frac{\log \prod_{i=1}^n M_i}{ \log M_{n+1}}=0.
$$
Hence Corollary~\ref{cor:diemension} applies and yields   
$$
\dim \cF= \liminf_{k\rightarrow \infty}\frac{\log N_k}{\log (1/\delta_k)} =\frac{\gamma-1/2-\tau}{\gamma+\tau},
$$
and the result follows from~\eqref{eq:subset} since $\tau>0$ is arbitrary.

\subsection{Proof of Theorem~\ref{thm:d=2}}  
The proof is similar  to the  proof of Theorem~\ref{thm:f}, so we only give a sketch.  
Let $f$ satisfy the conditions of Theorem~\ref{thm:d=2} and let $\cF_{\va, c}(f)$ denote the set of $x\in \cI$ such that 
$$
\left|\sum_{n=1}^{N}a_n\e(xf(n)) \right|\ge cN^{1/2} \quad \text{for infinitely many $N\in \N$}.
$$

Similarly to the proof of Theorem~\ref{thm:f}, first of all we  construct a   Cantor set inside $\cF_{\va, c}(f)$.

Fix a small parameter $\tau>0$. Let $L\in \N$ be a  large number  and let 
$$
N_1=\fl{\lambda(\cI) L^{\gamma-1-\tau}}.
$$
Divide $\cI$ into
$N_1$ subintervals of length 
$
|\cI|/N_1\gg L^{1-\gamma-\tau}
$
 which we denote as  $\cI_1, \ldots, \cI_{N_1}$.   

Applying Lemma~\ref{lem:aver f}  to each interval $\cI_k, 1\le k\le  N_1$, there exists $x_k\in \cI_k$ such that  
$$
\left |\sum_{n=L}^{2L}a_n\e\(x_k f(n)\) \right | \gg L^{1/2}.
$$   
Applying similar arguments to the proof of~\eqref{eq:QN}, we obtain 
$$
\max_{Q\le 2L} \left| \sum_{n=1}^{Q} a_n\e(x_k f(n)) \right| \gg L^{1/2}. 
$$
For each $x_k$, applying   Corollary~\ref{cor:max} and using that $f(t) \le t^{\gamma+o(1)}$, we obtain that there exists an 
interval $\cJ_k\subseteq \cI_k$ with length $|\cJ_k|=L^{-\gamma-\tau}$ such that 
$$
\max_{Q\le 2L} \left| \sum_{n=1}^{Q} a_n\e(x f(n)) \right| \gg L^{1/2}, \qquad \forall x\in \cJ_k.
$$

Note that  the collection of intervals $\cJ_k\subseteq \cI_k$, $1\le k\le N_1$, forms an $\cI(N_1, N_1, L^{-\gamma-\tau})$-pattern as in Definition~\ref{def:p}. Furthermore, this is the first step of the construction of the desired  Cantor-like set, and we denote the union of these intervals $\cJ_k, 1\le k\le N_1$, as $\cC_1$.

Let $L_k, k\in \N$ be a rapidly increasing sequence of numbers, for instance  
\begin{equation}
\label{eq:N+k+1}\log L_{k+1} \ge L_1L_2\ldots L_k.
\end{equation}

Suppose that we have constructed $k$-level Cantor set $\cC_k$ which is a collection of disjoint intervals with equal length $\delta_k$. Let 
$
N_{k+1}=\fl{\delta_k L_{k+1}^{\gamma-1-\tau}}
$
and for each  $\cJ\in \cC_k$ we  divide the interval  $\cJ\in \cC_k$  into $N_{k+1}$ subintervals of length 
$$
|\cJ|/N_{k+1}\gg L_{k+1}^{1-\gamma+\tau}.
$$
Applying the same argument as above to the interval $\cJ$,  there exists a $\cJ(N_{k+1}, N_{k+1}, \delta_{k+1})$-pattern $\cA\subseteq \cJ$ such that   
\begin{equation}
\label{eq:delta}
 \delta_{k+1}=L_{k+1}^{-\gamma-\tau},
\end{equation}
and
$$
\max_{Q\le 2L_{k+1}}\left |\sum_{n=1}^{Q}a_n\e\(x f(n)\) \right | \gg Q^{1/2}, \qquad \forall \ x\in \cA.
$$

Let $\cC_{k+1} $ be a collection of the $\cJ(N_{k+1}, N_{k+1}, \delta_{k+1})$-patterns inside each interval  $\cJ\in \cC_k$, see Remark~\ref{rem:J-patt}. 
The desired Cantor  set is defined as  
$$
\cC=\bigcap_{k=1}^{\infty} \cC_k.
$$  

Note that for some small constant $c>0$ the Cantor-like set $\cC$ is subset of $\cF_{\va, c}(f)$.  

By~\eqref{eq:N+k+1} and~\eqref{eq:delta} we conclude that for each $k\in \N$ the set $\cC_{k+1}$ contains 
$$\prod_{i=1}^{k+1}N_i=L_{k+1}^{\gamma-1-\tau+o(1)}
$$
 intervals with equal length  
$$
\delta_{k+1}=L_{k+1}^{-\gamma-\tau}.
$$
Combining with Corollary~\ref{cor:diemension} and the arbitrary choice of $\tau>0$ we conclude that 
$$
\dim \cC  \ge 1-1/\gamma, 
$$
which finishes the proof.

\section{Some heuristics on the Hausdorff dimension of the sets of  large sums} 
\label{sec:heuristic}
We start with the case of monomial sums. In particular, recall the notation~\eqref{eq:sigmadef} and~\eqref{eq:sigmaFdef}.  
It is natural to assume that $\sigma_d(x; N) $ is large only if $x$ can be well  approximated by a rational
number   with a  reasonably small denominator, that is,  belongs to {\it major arcs\/} in the traditional 
terminology, see~\cite{Vau}. While qualitatively this is an established fact, its optimal quantitive version is 
still unclear. Here we base our heuristics  on an approximate formula 
of Vaughan~\cite[Theorem~4.1]{Vau}.  More precisely, if 
$$
x =\frac{a}{q} + \xi
$$ 
for  some   integers $a$ and $q \ge 1$ with $\gcd(a,q) =1$ 
then 
\begin{equation}
\label{eq:VauAppr}
\begin{split}
\sigma_d(x; N) =\frac{1}{q} \sigma_d(a/q; q) & \int_0^{N} \e\(\xi \gamma^d\) d \gamma\\
& \qquad + O\(q^{1/2 + o(1)} \(1+ |\xi| N^d\)^{1/2}\).
\end{split}
\end{equation} 
It is also shown in~\cite{BD} that the error term is close to optimal. 
First we observe that   if $\xi< 0.5 N^{-d}$ then 
$$
\left| \int_0^{N} \e\(\xi \gamma^d\) d \gamma \right| \gg N. 
$$
Now assuming that ``typically'' we have $ \sigma_d(a/q; q)=q^{1/2 + o(1)}$, we conclude that 
$$
\left|\sigma_d(x; N) \right|  \ge N q^{-1/2 + o(1)} + O\(q^{1/2 + o(1)}\). 
$$
For any $\alpha > 1/2,$ setting $N = \fl{q^{1/2(1-\alpha) + \varepsilon}} $ we
obtain that for any $x \in \TTT $ such that  
\begin{equation}
\label{eq:Well Approx x}
\left| x -\frac{a}{q} \right| < q^{-d/2(1-\alpha) - d \varepsilon}
\end{equation}
holds for infinitely many  $a$ and $q \ge 1$ with $\gcd(a,q) =1$, 
we have 
$$
\left|\sigma_d(x; N) \right|  \ge N^{\alpha} 
$$  
for infinitely many $N$.  The argument in the proof of the classical Jarn\'ik--Besicovitch theorem,  see~\cite[Theorem~10.3]{Falconer}, implies that 
the set of $x \in \TTT $ satisfying $|x-a/q| \le q^{-\kappa}$ for infinitely many irreducible fractions $a/q$
with some fixed $\kappa \ge 2$, 
is of Hausdorff dimension $2/\kappa$. Hence, recalling~\eqref{eq:Well Approx x}, it seems reasonably to conjecture that 
$$
\dim \sF_{d, \alpha} =   \frac{4(1-\alpha)}{d}.
$$  
In particular, compared with~\eqref{eq:Hd=1}, for $d\ge 3$  this suggests that there is a discontinuity in the behaviour 
of  $\dim \sF_{d, \alpha}$ as a function of $\alpha$, most likely at $\alpha=1/2$.  

In principle similar arguments also apply to  $\sE_{d, \alpha}$ 
and may also lead to a conjecture about $\dim \sE_{d, \alpha}$. 
Instead of~\eqref{eq:VauAppr}, we now recall a 
result of Baker~\cite[Lemma~4.4]{Bak} which asserts that if for $\vx\in \Tor$ 
we have 
\begin{equation}
\label{eq:Approx}
x_i -\frac{a_i}{q}  = \xi_i
\end{equation} 
with some integers $a_1, \ldots, a_d$ and  $q\ge 1$ and real numbers
\begin{equation}
\label{eq:Small xi}
|\xi_i| \le \frac{1}{2d^2 q N^{i-1}}, \qquad  i =1, \ldots d, 
\end{equation}
then 
\begin{equation}
\label{eq:BakAppr}
\begin{split}
  S_d(\vx; N) 
&   =\frac{1}{q} S_d(\va/q; q)
%%\sum_{t=1}^q \e\(\(a_1t + \ldots a_dt^d\)/q\) 
\int_0^{N} \e\(\xi_d \gamma^d+\ldots +\xi_1\gamma\) d \gamma\\
& \qquad \qquad \qquad \qquad \qquad \quad  + O\(q^{1-1/d + o(1)} D^{1/d}\), 
\end{split}
\end{equation}
where 
$$
\va = (a_1, \ldots, a_d) \mand D =\gcd(a_2, \ldots, a_d,q).
$$
We now assume that  for all but a negligible set of  $\vx\in \Tor$ 
(say, of Hausdorff dimension zero) the following holds:
\begin{itemize}
\item  the corresponding exponential sums have square root cancellation, which holds, if for example
 the denominators $q$ are essentially square-free up to a factor of size $q^{o(1)}$; 
\item we have $D = q^{o(1)}$.
\end{itemize}
These are the main heuristic assumptions of our approach. 
Under these assumptions, analysing the proof of~\eqref{eq:BakAppr} in~\cite{Bak},  we see that~\eqref{eq:BakAppr}  can heuristically be transformed into
\begin{equation}
\label{eq:BakApprHeurist}
\begin{split}
 S_d(\vx; N)&   =\frac{1}{q} S_d(\va/q; q)  \int_0^{N}  \e\(\xi_d \gamma^d+\ldots +\xi_1\gamma\) d \gamma\\
& \qquad \qquad \qquad \qquad \qquad \qquad \qquad +O\(q^{1/2+ o(1)} \). 
\end{split}
\end{equation}

Furthermore,  if   
$$
|\xi_i| \le \frac{1}{2d^2 N^{i}}, \qquad  i =1, \ldots d, 
$$
then 
$$
\left| \int_0^{N} \e\(\xi \gamma^d\) d \gamma \right| \gg N, 
$$
and thus 
$$
\left| S_d(\vx; N) \right|  \ge N q^{-1/2 + o(1)} + O\(q^{1/2 + o(1)}\), 
$$
provided that 
 $$
\left| x_i  -\frac{a}{q} \right| <  \frac{1}{2d^2  N^i } \le  \frac{1}{2d^2 q N^{i-1}},  \qquad  i =1, \ldots d, 
$$
(ignoring a very small set of $\vx \in \Tor$).

For any $1/2<\alpha <1$, setting  $N = \fl{q^{1/2(1-\alpha) + \varepsilon}} $, we obtain that for any $\vx\in \T_d$ such that 
there are infinitely many approximations   
$$
\left| x_i  -\frac{a_i}{q} \right| <  q^{-i/2(1-\alpha)-i \varepsilon}, \qquad  i =1, \ldots d, 
$$
%%then 
we have  
$|S_d(\vx; N)|\ge N^{\alpha}$
for infinitely many $N$.

Let $\cX_{d,\alpha}$ be  the set of   $\vx\in \Tor$ such that  there are infinitely many 
approximations 
$$
\left| x_i  -\frac{a_i}{q} \right| <  q^{-i/2(1-\alpha)}, \qquad  i =1, \ldots d. 
$$
This naturally leads us to the conjecture that  
$$
\dim \sE_{d, \alpha} = \dim \cX_{d, \alpha}.
$$

We also consider the set   $\cX_{d,\alpha}^\sharp \subseteq \cX_{d,\alpha}$, which is defined exactly as $\cX_{d,\alpha}$ 
with the additional condition that the denominator  $q= p$ is prime. That is, $\cX_{d,\alpha}^\sharp$ is
the set of   $\vx\in \Tor$ such that  there are infinitely many 
approximations 
$$
\left| x_i  -\frac{a_i}{p} \right| <  p^{-i/2(1-\alpha)}, \qquad  i =1, \ldots d, 
$$
with a prime $p$. 

We also note that for $\vx \in  \Tor$ such that  with a prime $q=p$ we have~\eqref{eq:Approx}
and~\eqref{eq:Small xi},  using
 the {\it Weil bound\/}, see, for example,~\cite[Chapter~6, Theorem~3]{Li}, in the argument of
 the proof of~\cite[Lemma~4.4]{Bak}, the 
asymptotic formula~\eqref{eq:BakApprHeurist} can be established 
rigorously with $p$ instead  of $q$.

We also recall that by a result of Knizhnerman and  Sokolinskii~\cite[Theorem~1]{KnSok1}, see also~\cite{KnSok2}, there is positive 
proportion of rational exponential sums $S_d(\va/p; p)$,  which are of order $p^{1/2}$,
that is, with 
\begin{equation}
\label{eq:large-lower}
\left| S_d(\va/p; p) \right| \gg p^{1/2}.
\end{equation}
Furthermore, by~\cite[Lemma~2.6]{ChSh-AM},  the  corresponding coefficients $\va/p = (a_1/p, \ldots, a_d/p)$ are densely distributed  in the cube $[0,1]^d$.  
 %integer coefficients $\va = (a_1, \ldots, a_d)$ are 
%densely distributed  in the cube $[0, p-1]^d$. 
This shows that the main term in~\eqref{eq:BakApprHeurist}
is %sufficiently 
large for a large subset of $\vx \in \Tor$. We remark that  for $d=2$, that is, for Gauss sums,  
the bound~\eqref{eq:large-lower} holds for all $(a_1/p, a_2/p)$ with $a_2\neq 0$,  see~\cite[Equation~(1.55)]{IwKow}.

Hence, using the above observations,  one can perhaps produce  a rigorous argument that 
$$
\dim \sE_{d, \alpha} \ge  \dim  \cX_{d,\alpha}^\sharp.
$$ 
Applying a result of Rynne~\cite[Theorem~1]{Rynne} we obtain 
$$
 \dim  \cX_{d,\alpha}^\sharp = \dim \cX_{d,\alpha}=\mathfrak{s}(d, \alpha), 
$$ 
 where  
$$
\mathfrak{s}(d, \alpha)=\min_{j=1, \ldots, d}\frac{d+1+j\vartheta_j -\sum_{i=1}^{j} \vartheta_i}{1+\vartheta_j} ,
$$
and
$$
\vartheta_i=\frac{i}{2(1-\alpha)}-1, \qquad i=1, \ldots, d.
$$ 

We remark that the condition $\alpha\ge 1/2$ makes sure  the assumption of~\cite[Theorem~1]{Rynne} holds, that is, we have 
$$
\sum_{i=1}^d \vartheta_i \ge 1.
$$  
For a different approach to $\dim \cX_{d, \alpha}^\#$ and $\dim \cX_{d, \alpha}$, see also~\cite[Corollary~5.1]{Wang}.

We recall that the upper bound of $\dim \sE_{d, \alpha}$ in~\cite[Theorem~1]{ChSh-JNT}  claims that for $d\ge 2$ and $\alpha\in (1/2, 1)$ one has 
$$
\dim \sE_{d, \alpha} \le  \mathfrak{u}(d, \alpha),
$$
where  
$$
\mathfrak{u}(d, \alpha)=\min_{k=0, \ldots, d-1} \frac{(2d^{2}+4d)(1-\alpha)+k(k+1)}{4-2\alpha+2k}.
$$

We now compare the  values of $\mathfrak{s}(d, \alpha)$ and $\mathfrak{u}(d, \alpha)$ for   $d=2$.   We have 
$$
\mathfrak{s}(2, \alpha)= 
  \begin{cases}
 \displaystyle \frac{7-6\alpha}{2}   & \text{for } 1/2\le  \alpha\le 5/6,\\
   6(1-\alpha) &   \text{for } 5/6 <\alpha<1,
  \end{cases}
$$
and 
$$
\mathfrak{u}(2, \alpha)= 
  \begin{cases}
  \displaystyle \frac{9-8\alpha}{3-\alpha}   &  \text{for } 1/2\le  \alpha\le 6/7,\\ 
 \displaystyle   \frac{8(1-\alpha)}{2-\alpha} &  \text{for }  6/7 <\alpha<1.
  \end{cases}
$$
Note that for the endpoints $\alpha=1/2$ and $\alpha= 1$ we have 
$$
\mathfrak{s}(2, 1/2)=\mathfrak{u}(2, 1/2)=2,
$$
and 
$$
\mathfrak{s}(2, 1)=\mathfrak{u}(2, 1)=0.
$$
Moreover, it is somewhat tedious but elementary to derive that  
$$
\mathfrak{s}(2, \alpha)<\mathfrak{u}(2, \alpha), \qquad 1/2<\alpha<1.
$$ 

For the case  $d\ge 3$ and for the value $\mathfrak{u}(d, \alpha )$  we  have 
$$
\mathfrak{u}(d, 1/2)=d \mand \mathfrak{u}(d, 1)=0.
$$
However, for the value  $\mathfrak{s}(d, 1/2)$ we have  
$$
\mathfrak{s}(d, 1/2)=\min_{j=1, \ldots, d} \frac{2(d+1)+j^{2}-j}{2j}.
$$ 
Thus  we have 
$$
\lim_{d \to \infty} \mathfrak{s}(d, 1/2)/ \sqrt{2d} = 1.
$$
In particular,  compared with~\eqref{eq:Hd=1}   this suggests that, as in the case of  $\dim \sF_{d, \alpha}$, there is a discontinuity in the behaviour of $\dim \sE_{d, \alpha}$, most likely at $\alpha =1/2$ when $d\ge 3$.

\section*{Acknowledgement}

During the preparation of this work, 
C.C. was  supported   by the Hong Kong Research Grants Council GRF Grants  CUHK14301218 and CUHK14304119.
C.C. and I.S.  were   supported  
by the Australian Research Council Grant~DP170100786. B.K. was supported  by the Australian Research Council Grant DP160100932 and Academy of 
Finland Grant~319180.  J.M. was supported by a Royal Society Wolfson Merit Award, and funding from the European Research Council (ERC) under the European Union's     Horizon 2020 research and innovation programme (grant agreement No 851318).


\begin{thebibliography}{99}




\bibitem{Bak} R. C. Baker,   {\it Diophnatine inequalities\/},   Oxford Univ. Press, 1986. 
 
\bibitem{Bomb} E. Bombieri, `On exponential sums in finite fields',
{\it Amer. J. Math.\/},    {\bf 88} (1966), 71--105.

\bibitem{BDG} J. Bourgain, C. Demeter and L. Guth, 
`Proof of the main conjecture in Vinogradov's mean value theorem for degrees higher than three', 
{\it Ann.\ Math.\/}, {\bf 184} (2016), 633--682. 

\bibitem{Brow1} T. D. Browning, `Equal sums of two $k$th powers', 
{\it  J. Number Theory\/}, {\bf 96} (2002),   293--318. 

\bibitem{Brow2} T. D. Browning, `Sums of four biquadrates', 
{\it  Math. Proc. Cambridge Philos. Soc.\/}, {\bf 134} (2003),  385--395. 

\bibitem{BrH-B} T. D. Browning and D. R. Heath-Brown, 'Plane curves in boxes and equal sums of two powers', 
{\it Math. Zeit.\/}, {\bf 251} (2005),   233--247. 


\bibitem{Brud} J. Br\"udern, `Approximations to Weyl sums', {\it Acta Arith.\/}, {\bf 184} (2018), 287--296.

\bibitem{BD} J. Br\"udern and D. Daemen, `Imperfect mimesis of Weyl sums',  {\it Internat. Math. Res. Notices\/},, {\bf 2009} (2009), 3112--3126.
 
\bibitem{Cas}
J. W. S. Cassels
`Some metrical theorems in Diophantine approximation. I',  {\it Proc. Cambridge Philos. Soc.}, {\bf 46} (1950), 209--218.
  
\bibitem{ChSh-AM} C. Chen and I. E. Shparlinski,
`On large values of Weyl sums',  
{\it Adv. Math.\/},  {\bf 370} (2020). Article 107216.


\bibitem{ChSh-JNT} C. Chen and I. E. Shparlinski, `Hausdorff dimension of the large values of Weyl sums',
{\it J. Number Theory\/},  {\bf 214} (2020) 27--37.

\bibitem{ChSh-IMRN} C. Chen and I. E. Shparlinski,
`New bounds of Weyl sums',  
{\it   Intern. Math. Res. Notices\/},    (to appear).

\bibitem{CG} J. Cilleruelo and A. Granville, `Lattice points on circles, squares in arithmetic progressions, and sumsets of
squares', {\it Additive Combinatorics, CRM Proceedings $\&$ Lecture Notes\/}, {\bf 43} (2007),  241--262.
 
\bibitem{DrTi} M.\ Drmota and R.\ Tichy,
{\it Sequences, discrepancies and applications\/},
Springer-Verlag, Berlin, 1997.

%\bibitem{Dev} J. Dever, `Local Hausdorff measure',  
%{\it Preprint}, 2016, available at \url{https://arxiv.org/abs/1610.00078}. 

\bibitem{EG} L. C. Evans and R. F. Gariepy, {\it Measure theory and fine properties of
functions\/}. Boca Raton, FL, CRC, 1992.

\bibitem{Falconer} K. J. Falconer, {\it Fractal geometry: Mathematical foundations and applications\/},
John Wiley, 2nd Ed., 2003.

\bibitem{FK} A. Fedotov and F. Klopp, `An exact renormalization formula for Gaussian exponential sums and
applications',  {\it Amer. J. Math.\/}, {\bf 134} (2012), 711--748.


\bibitem{FWW} D. J. Feng, Z. Y. Wen and J. Wu, `Some dimensional results for homogeneous Moran sets',  {\it Sci. China Ser.
A.\/}, {\bf 40} (1997), 475--482.

\bibitem{FJK} H. Fiedler, W. Jurkat and O. K\"orner, `Asymptotic expansions of finite theta series', {\it Acta Arith.\/}, {\bf 32}
(1977), 129--146.

%\bibitem{Ford}
%K. B. Ford, `Vinogradov's  integral  and  bounds  for  the  Riemann  zeta  function', 
%{\it Proc. London Math. Soc}, (3) {\bf 85} (2002), 565--633.

%\bibitem{JuSta}
% H. J{\"u}rgensen and L. Staiger, 
% `Local Hausdorff dimension', {\it Acta Inform.\/}, {\bf  32} (1995),
%491--507.
\bibitem{Gall}
P. Gallagher, `Approximation by reduced fractions', {\it J. Math. Soc. Japan\/}, {\bf 13}, (1961), 342--345.

\bibitem{HL1}
 G. H. Hardy and J. E. Littlewood, `The trigonometric series associated with the elliptic
$\vartheta$-functions', {\it Acta Math.\/}, {\bf 37} (1914), 193--239.
%G. H. Hardy and J. E. Littlewood, `Some problems of diophantine approximation',
% Collected papers of G. H. Hardy, Oxford, (1966), 193--238.


\bibitem{HL2}
G. H. Hardy and J. E. Littlewood, `Some problems of Diophantine approximation: A remarkable trigonometric series', 
{\it Proc. Nat. Acad. Sci.\/}, {\bf 2} (1916), 583--586.

%\bibitem{Harm} G. Harman, {\it Metric number theory\/}, 
%London Math. Soc. Monographs. New Ser., vol.~18,  The Clarendon Press, Oxford Univ. Press, 
%New York, 1998.

\bibitem{H-B1} D. R. Heath-Brown, `The density of rational points on cubic surfaces', {\it Acta Arith.\/}, {\bf 79} (1997), 17--30.

\bibitem{H-B2}  D. R. Heath-Brown, 'Counting rational points on algebraic varieties', 
{\it Analytic number theory\/},  
Lecture Notes in Math., vol.~1891, Springer, Berlin, 2006, 51--95,

\bibitem{Hool1} 
C. Hooley, `On another sieve method and the numbers that are a sum of two $h$th
powers', {\it  Proc. London Math. Soc.\/}, {\bf 36} (1978),  117--140.  

\bibitem{Hool2} 
C. Hooley, `On another sieve method and the numbers that are a sum of two $h$th
powers, II', {\it  J. Reine Angew Math.\/}, {\bf 475} (1996),  55--75. 

\bibitem{IwKow} H. Iwaniec and E. Kowalski,
{\it Analytic number theory\/}, Amer.  Math.  Soc.,
Providence, RI, 2004. 

%\bibitem{KnSok1} L. A. Knizhnerman and V. Z. Sokolinskii, `Some estimates for rational trigonometric sums and sums of Legendre symbols', {\it Uspekhi Mat. Nauk\/}, {\bf 34} (3)  (1979), 199--200 (in Russian); 
%translated in {\it Russian Math. Surveys\/}, {\bf 34} (3) (1979), 203--204. 

%
%
%\bibitem{KRS} 
%A. K{\"a}enm{\"a}ki, T. Rajala  and V. Suomala, `Local multifractal analysis in metric spaces', 
%{\it   Nonlinearity\/}, {\bf  26}  (2013),  2157--2173.

% \bibitem{Kin}
%A. Khinchin,  {\it Continued fractions\/}, New York, Dover Publ., 1964.


\bibitem{KnSok1} L. A. Knizhnerman and V. Z. Sokolinskii, `Some estimates for rational trigonometric sums 
and sums of Legendre symbols', {\it Uspekhi Mat. Nauk\/}, {\bf 34} (3)  (1979), 199--200 (in Russian); 
translated in {\it Russian Math. Surveys\/}, {\bf 34} (3) (1979), 203--204. 

\bibitem{KnSok2} L. A. Knizhnerman and V. Z. Sokolinskii, `Trigonometric sums and sums of Legendre 
symbols with large and small absolute values',  {\it Investigations in Number Theory\/}, Saratov, Gos. Univ., 
Saratov, 1987,  76--89 (in Russian). 

 \bibitem{Li}  W.-C. W. Li, {\it Number theory with applications\/},
World Scientific,  Singapore, 1996.

\bibitem{Mattila1995} P. Mattila, {\it Geometry of sets and measures in Euclidean spaces: Fractals and
rectifiability\/}, Cambridge Univ. Press, 1995.
 
%\bibitem{Mattila2015}  P. Mattila,  {\it Fourier analysis and Hausdorff dimension\/}, 
%Cambridge Studies in Advanced Math.,  vol.~50, Cambridge Univ. Press, 2015.
 
%\bibitem{MV}
%H. L. Montgomery and R. C. Vaughan, `Hilbert's inequality', 
%{\it J. London Math. Soc.\/},  {\bf 8} (1974), 73--82. 

 
\bibitem{Morm}
O. Mormon,  `Sums  and differences of four $k$th powers', 
{\it Monat. Math.\/}, {\bf 164} (2011),  55--74.

%\bibitem{Mord}
%L. J. Mordell,  `On a sum analogous to a Gauss sum', 
%{\it  Quart. J. Math.\/}, {\bf 3} (1932),  161--167.

%
%\bibitem{Ols1} L. Olsen, `Applications of divergence points to local dimension functions of subsets
%of $R^d$', {\it Proc. Edinburgh Math. Soc.\/}, {\bf  48} (2005), 213--218.
%
%\bibitem{Ols2} L. Olsen, `Characterization of local dimension functions of subsets
%of $R^d$', {\it Colloq. Math.\/}, {\bf  103} (2005), 231--239.

\bibitem{Rud} W. Rudin, `Some theorems on Fourier coefficients', 
{\it  Proc. Amer Math. Soc.\/}, {\bf 10} (1959), 855--859.

\bibitem{Rynne} B. P. Rynne, `Hausdorff dimension and generalized simultaneous Diophantine approximation',  {\it Bull.
London Math. Soc.\/} {\bf 30} (1998), 365--376.


\bibitem{SkinWool} 
C. M. Skinner and T. D. Wooley, 'Sums of two $k$th powers', 
 {\it  J. Reine Angew Math.\/}, {\bf  462}  (1995), 57--68.
 
\bibitem{Step} S. A.    Stepanov,
`Rational trigonometric sums along a curve',  {\it Automorphic Functions and Number Theory, II\/}, 
 Zap. Nauchn. Sem. Leningrad. Otdel. Mat. Inst. Steklov., vol.~134 (1984), 232--251 (in Russian). 
 
 \bibitem{Vau} R. C. Vaughan, {\it The Hardy-Littlewood method\/},  Cambridge Tracts in Math. vol.~25, 
 Cambridge Univ. Press, 1997. 

\bibitem{Wang} B. Wang, J. Wu and J. Xu, `Mass transference principle for limsup sets generated by rectangles'. {\it  Math. Proc. Cambridge Philos. Soc.\/},  {\bf 158} (2015), 419--437.


%\bibitem{Watt}
%N. Watt, {\it Exponential sums and the Riemann zeta-function II}, J. London Math. Soc.
%{\bf 39}
%(1989), 385--404.

%\bibitem{Wool}
%T. D. Wooley, `On Vinogradov's mean value theorem', Mathematika {\bf 39} (1992), 379--399

%\bibitem{Wool1}
%T. D. Wooley, `On Vinogradov's mean value theorem, II', Michigan Math. J. {\bf 40} (1993), 175--180.

\bibitem{Wool2} T.~D.~Wooley,
`The cubic case of the main conjecture in Vinogradov's mean value theorem',  
{\it Adv.  in Math.\/}, {\bf  294} (2016), 532--561.

\bibitem{Wool3} T.~D.~Wooley, `Perturbations of Weyl sums', {\it Internat. Math. Res. Notices\/}, {\bf 2016} (2016), 2632--2646.



\bibitem{Wool5} T.~D.~Wooley, 
`Nested efficient congruencing and relatives of Vinogradov's mean value theorem',
{\it Proc. London Math. Soc.\/},  {\bf 118} (2019), 942--1016.


\end{thebibliography}
\end{document}